\documentclass[reqno]{amsart}
\usepackage{amssymb,eucal,latexsym, mathrsfs,xy,enumerate,graphicx,url}

\xyoption{all}

\newenvironment{enumeratei}{\begin{enumerate}[\upshape (i)]}%
{\end{enumerate}}
{\end{enumerate}}
\newenvironment{enumerater}{\begin{enumerate}[\upshape (1)]}%
{\end{enumerate}}

\usepackage[usenames]{color}

\newcommand{\pup}[1]{\textup{(}{#1}\textup{)}}

\newcommand{\oo}[1]{\left]#1\right[}

\newcommand{\eqdef}{\underset{\mathrm{def}}{=}}
\newcommand{\cp}{\smallsetminus}
\newcommand{\cpt}{\mathbin{\widetilde{\smash{\smallsetminus}}}}

\DeclareMathOperator{\Idt}{Idt}
\newcommand{\Sim}[1]{\operatorname{Sim}({#1})}
\DeclareMathOperator{\Seq}{Seq}
\DeclareMathOperator{\Seqr}{Seq_{red}}
\DeclareMathOperator{\lh}{lh}

\DeclareMathOperator{\VR}{Vert}

\newcommand{\INF}[2]{{#1}^{\leqslant{#2}}}
\newcommand{\SUP}[2]{{#1}^{\geqslant{#2}}}

\newcommand{\Fm}[1]%
{\mathrm{F}_{\mathrm{mon}}({#1})}

\newcommand{\Fg}[1]%
{\mathrm{F}_{\!\mathrm{gp}}({#1})}

\newcommand{\Umn}{\operatorname{U_{mon}}}
\newcommand{\Um}[1]{\Umn({#1})}

\newcommand{\Ugn}{\operatorname{U_{gp}}}
\newcommand{\Ug}[1]{\Ugn({#1})}

\newcommand{\Catn}{\operatorname{Cat}}
\newcommand{\Cat}[1]{\Catn({#1})}

\newcommand{\Path}[1]{\operatorname{Path}({#1})}

\newcommand{\hgn}{\Pi_1}
\newcommand{\hg}[1]{\hgn({#1})}

\newcommand{\HGN}{\Upsilon^{\pm}}
\newcommand{\HG}[1]{\HGN({#1})}

\newcommand{\HM}[1]{\Upsilon({#1})}

\newcommand{\lcm}{\vee}
\newcommand{\lcmt}{\mathbin{\widetilde\lcm}}
\renewcommand{\gcd}{\wedge}
\newcommand{\gcdt}{\mathbin{\widetilde\wedge}}
\newcommand{\Gcd}{\bigwedge}
\newcommand{\Gcdt}{\mathbin{\widetilde\bigwedge}}
\newcommand{\Lcm}{\bigvee}
\newcommand{\Lcmt}{\mathbin{\widetilde\bigvee}}

\newcommand{\red}[1]{{#1}_{\mathrm{red}}}

\newcommand{\gf}{\varphi}

\newcommand{\gs}{\sigma}

\newcommand{\gO}{\Omega}
\newcommand{\gS}{\Sigma}

\newcommand{\js}{join-semi\-lat\-tice}
\newcommand{\ms}{meet-semi\-lat\-tice}

\newcommand{\one}{\mathbf{1}}

\newcommand{\ol}[1]{\overline{#1}}

\newcommand{\pI}[1]{\bigl({#1}\bigr)}

\newcommand{\set}[1]{\left\{#1\right\}}
\newcommand{\setm}[2]{\set{{#1}\mid{#2}}}
\newcommand{\vecm}[2]{\left({#1}\mid{#2}\right)}

\newcommand{\eps}{\varepsilon}

\newcommand{\es}{\varnothing}

\newcommand{\CAT}{\mathbf{Cat}}
\newcommand{\MON}{\mathbf{Mon}}
\newcommand{\GP}{\mathbf{Gp}}
\newcommand{\GPD}{\mathbf{Gpd}}
\newcommand{\PO}{\mathbf{Pos}}

\newcommand{\NN}{\mathbb{N}}
\newcommand{\ZZ}{\mathbb{Z}}

\newcommand{\conc}{\mathbin{\stackrel{\smallfrown}{}}}

\newcommand{\cC}{{\mathcal{C}}}

\numberwithin{equation}{section}

\newtheorem*{stat}{\name}
\newcommand{\name}{testing}

\newenvironment{all}[1]{\renewcommand{\name}{#1}\begin{stat}}
                        {\end{stat}}

\theoremstyle{plain}

\newtheorem{theorem}{Theorem}[section]
\newtheorem{proposition}[theorem]{Proposition}
\newtheorem{corollary}[theorem]{Corollary}
\newtheorem{lemma}[theorem]{Lemma}
\newtheorem{examplepf}[theorem]{Example}

\newtheorem*{sclaim}{Claim}

\theoremstyle{definition}

\newtheorem{definition}[theorem]{Definition}
\newtheorem{notation}[theorem]{Notation}
\newtheorem{example}[theorem]{Example}

\theoremstyle{remark}
\newtheorem{remark}[theorem]{Remark}

\newcommand{\qedc}{{\qed}~{\rm Claim~{\theclaim}.}}
\newcommand{\qedsc}{{\qed}~{\rm Claim.}}

\newenvironment{scproof}
{\begin{proof}[Proof of Claim.]}
{\qedsc\renewcommand{\qed}{}\end{proof}}

\numberwithin{figure}{section}
\numberwithin{table}{section}

\newcommand{\ba}{\boldsymbol{a}}
\newcommand{\bb}{\boldsymbol{b}}
\newcommand{\bc}{\boldsymbol{c}}

\newcommand{\bu}{\boldsymbol{u}}
\newcommand{\bv}{\boldsymbol{v}}

\newcommand{\bx}{\boldsymbol{x}}
\newcommand{\by}{\boldsymbol{y}}
\newcommand{\bz}{\boldsymbol{z}}

\newcommand{\sr}[1]{\partial_0{#1}}
\newcommand{\tg}[1]{\partial_1{#1}}
\newcommand{\SRn}{\nabla_{\!0}}
\newcommand{\TGn}{\nabla_{\!1}}
\newcommand{\SR}[1]{\SRn{#1}}
\newcommand{\TG}[1]{\TGn{#1}}

\newcommand{\dive}{\leqslant}
\newcommand{\ndive}{\not\leqslant}
\newcommand{\divet}{\mathbin{\widetilde{\leqslant}}}
\newcommand{\ndivet}{\mathbin{\widetilde{\not\leqslant}}}

\title{Gcd-monoids arising from homotopy groupoids}

\author[F. Wehrung]{Friedrich Wehrung}
\address{LMNO, CNRS UMR 6139\\
D\'epartement de Math\'ematiques\\
Universit\'e de Caen Normandie\\
14032 Caen cedex\\
France}
\email{friedrich.wehrung01@unicaen.fr}
\urladdr{https://wehrungf.users.lmno.cnrs.fr}


\date{\today}

\begin{document}

\begin{abstract}
The \emph{interval monoid}~$\HM{P}$ of a poset~$P$ is defined by generators $[x,y]$, where $x\leq y$ in~$P$, and relations $[x,x]=1$, $[x,z]=[x,y]\cdot[y,z]$ for $x\leq y\leq z$.
It embeds into its universal group~$\HG{P}$, the \emph{interval group} of~$P$, which is also the universal group of the homotopy groupoid of the chain complex of~$P$.
We prove the following results:

\begin{itemize}
\item The monoid~$\HM{P}$ has finite left and right greatest common divisors of pairs (we say that it is a \emph{gcd-monoid}) if{f} every principal ideal (resp., filter) of~$P$ is a \js\ (resp., a \ms).

\item For every group~$G$, there is a connected poset~$P$ of height~$2$ such that~$\HM{P}$ is a gcd-monoid and~$G$ is a free factor of~$\HG{P}$ by a free group.
Moreover, $P$ can be taken to be finite if{f}~$G$ is finitely presented.

\item For every finite poset~$P$, the monoid~$\HM{P}$ can be embedded into a free monoid.

\item Some of the results above, and many related ones, can be extended from interval monoids to the universal monoid~$\Um{S}$ of any category~$S$.
This enables us, in particular, to characterize the embeddability of~$\Um{S}$ into a group, by stating that it holds at the hom-set level.
We thus obtain new easily verified sufficient conditions for embeddability of a monoid into a group.
\end{itemize}
We illustrate our results by various examples and counterexamples.
\end{abstract}

\maketitle

\section{Introduction}\label{S:Intro}
\subsection{Background and highlights}
\label{Su:BackHigh}
The results of the present paper are, basically, easy.
Nonetheless, the purpose they fill is non-trivial, as they lay the foundation for the study of a class of monoids illustrating (cf. Dehornoy and Wehrung~\cite{DW1}) the limitations of an approach, initiated in Dehornoy~\cite{Dis}, of the word problem in certain groups.
The class of groups in question includes Artin-Tits groups, and, more generally, universal groups of cancellative monoids with left and right greatest common divisors --- in short \emph{gcd-monoids} (cf. Dehornoy~\cite{Dis}).

Those monoids, together with the following construct, are our main object of study.
The \emph{interval monoid}~$\HM{P}$ of a poset (i.e., partially ordered set)~$P$ is the universal monoid of the category associated with~$P$ in the usual manner, and its universal group~$\HG{P}$ is also the universal group of the homotopy groupoid of the chain complex of~$P$.
Accordingly, we will call~$\HG{P}$ the \emph{floating homotopy group} of~$P$ (cf. Definition~\ref{D:FHG}).

The monoids~$\HM{P}$ are fairly special objects:
for example, $\HM{P}$ always embeds into its universal group~$\HG{P}$ (cf. Proposition~\ref{P:rGPintoGrp}), and further, if~$P$ is finite, then~$\HM{P}$ always embeds into a free monoid (Proposition~\ref{P:rGpLinExt}).
Nevertheless, every group is a free factor, by a free group, of the universal group of a gcd-monoid of the form~$\HM{P}$ (see Theorem~\ref{T:VG2rGL2} below).

The concept of a gcd-monoid, which is implicit in the approach
of braids and Artin-Tits monoids by Garside~\cite{Gars1969} and Brieskorn-Saito~\cite{BriSai1972},
and underlies all that is now called Garside theory (cf. Dehornoy~\cite{DDGKM}), was introduced as such in Dehornoy~\cite{Dis}.

The present paper sets the foundation of the study of the interaction between groups, gcd-monoids, and interval monoids.
Three highlights are the following.

\begin{all}{Theorem \ref{T:UmSgcdmon}}
The universal monoid of a category~$S$ is a gcd-monoid if{f}~$S$ is cancellative, has no non-trivial left invertibles, and has left gcds \pup{resp., right gcds} of any pair of elements with the same source \pup{resp., target}.
\end{all}

In particular, for a poset~$P$, the condition that~$\HM{P}$ be a gcd-monoid can be easily read on~$P$, by saying that every principal ideal is a \js\ and every principal filter is a \ms\ (Proposition~\ref{P:rGpPgcd}).
This condition holds if~$P$ is the barycentric subdivision of a simplicial complex (Corollary~\ref{C:rGpPgcd}).

\begin{all}{Theorem~\ref{T:VG2rGL2}}
Every group~$G$ is a free factor, by a free group, of the universal group~$\HG{P}$ of~$\HM{P}$, for some poset~$P$ of height~$2$ such that~$\HM{P}$ is a gcd-monoid.
Furthermore, $P$ can be taken to be finite if{f}~$G$ is finitely presented.
\end{all}

\begin{all}{Theorem \ref{T:EmbUmS2Grp}}
The universal monoid of a category~$S$ embeds into a group if{f} there is a functor, from~$S$ to a group, which separates every hom-set of~$S$.
\end{all}

\subsection{Section by section summary of the paper}
\label{Su:SectSumm}

We start by recalling, in \textbf{Section~\ref{S:Basic}}, the basic concepts and notation underlying the paper.

For any category~$S$, described in an ``arrow-only'' fashion \emph{via} its partial semigroup of arrows, the \emph{universal monoid}~$\Um{S}$ of~$S$ is the initial object in the category of all functors from~$S$ to some monoid.
The \emph{universal group}~$\Ug{S}$ of~$S$ is defined similarly for groups.
In \textbf{Section~\ref{S:UnivMon}}, we present an overview of those concepts, essentially originating in Higgins~\cite{Higg1971}.

Say that a category is \emph{conical} if it has no non-trivial left (equivalently, right) invertibles.
The main goal of \textbf{Section~\ref{S:UmScanc}} is to prove that a category~$S$ is conical (resp., cancellative) if{f} the monoid~$\Um{S}$ is conical (resp., cancellative).

In \textbf{Section~\ref{S:divetgcd}}, we describe the left and right divisibility preorderings on~$\Um{S}$ and we prove that~$\Um{S}$ is a gcd-monoid if{f}~$S$ has left (resp., right) gcds of pairs with the same source (resp., target).

In \textbf{Section~\ref{S:FHG}}, we study the ``floating homotopy group~$\HG{K}$'' of a simplicial complex~$K$, observing in particular that it is the universal group of the homotopy groupoid~$\hg{K}$ of~$K$.
We observe in Proposition~\ref{P:VG2rGL} that for a spanning tree with set of edges~$E$ in a connected simplicial complex~$K$, $\HG{K}$ is the free product of the fundamental group of~$K$ and the free group on~$E$.

The interval monoid~$\HM{P}$, of a poset~$P$, is defined, in \textbf{Section~\ref{S:IntMon}}, as the universal monoid of the category~$\Cat{P}$ canonically associated to~$P$ (i.e., there is exactly one arrow from~$x$ to~$y$ if{f} $x\leq y$, no arrow otherwise).
We verify that the universal group~$\HG{P}$ of~$\HM{P}$ is identical to the floating homotopy group~$\HG{\Sim{P}}$ of the chain complex~$\Sim{P}$ of~$P$, thus showing that there is no ambiguity on the notation~$\HG{P}$.
We also verify that the canonical map from~$\HM{P}$ to~$\HG{P}$ is one-to-one --- of course this result does not extend to arbitrary categories.
We characterize those posets~$P$ such that~$\HM{P}$ is a gcd-monoid, and we show that this condition always holds if~$P$ is the barycentric subdivision of a simplicial complex.

In \textbf{Section~\ref{S:SubMonFM}} we verify that~$\HM{P}$ always embeds into a free monoid (not only a free group) if~$P$ is finite, but that nonetheless, we construct a submonoid of the free monoid on four generators, which is a gcd-monoid, but not the interval monoid of any poset.

In \textbf{Section~\ref{S:Spindle}}, we introduce a class of monoids, denoted by $\HM{P,u,v}$, where $[u,v]$ is a so-called \emph{extreme spindle} of a poset~$P$.
The monoids $\HM{P,u,v}$ are also constructed as universal monoids of categories, and they are used as counterexamples in the follow-up paper Dehornoy and Wehrung~\cite{DW1}.
A presentation of $\HM{P,u,v}$ consists of a certain subset of the natural presentation of~$\HM{P}$ (cf. Proposition~\ref{P:SpdMonPres}), and if~$\HM{P}$ is a gcd-monoid, then so is $\HM{P,u,v}$ (cf. Proposition~\ref{P:SpindleCat}).

In \textbf{Section~\ref{S:Pos2UMG}} we show that for any category~$S$, the universal monoid~$\Um{S}$ can be embedded into a group if{f} the embedding can be realized at the hom-set level.
We illustrate that result on an example, which we denote by~$C_6$ (Example~\ref{Ex:EmbGrp1}).
While neither Adjan's condition nor Dehornoy's $3$-Ore condition are \emph{a priori} applicable to prove the embeddability of~$C_6$ into a group, this example is the universal monoid of a category, so the above-mentioned criterion easily applies.

\section{Notation and terminology}
\label{S:Basic}

The present paper will involve \emph{small categories}, viewed in an ``arrow-only'' fashion as partial semigroups with identities subjected to certain axioms (cf. Definition~\ref{D:Semicat}), as well as categories viewed in the usual object / arrow way, which are then usually proper classes.
In order to highlight the distinction without dragging along the qualifier ``small'' through the paper, we will denote the first kind (small categories, usually arrow-only) as ``categories'', and the second kind (object / arrow categories) as ``Categories''.
Similar conventions will apply to groupoids versus Groupoids (categories where every arrow is an isomorphism) and functors versus Functors (morphisms of categories).

The following Categories will be of special importance:
\begin{itemize}
\item $\CAT$, the Category of all categories (the morphisms are the functors),

\item $\GPD$, the Category of all groupoids (the morphisms are the functors),

\item $\MON$, the Category of all monoids (the morphisms are the monoid homomorphisms),

\item $\GP$, the Category of all groups (the morphisms are the group homomorphisms),

\item $\PO$, the Category of all posets (i.e., partially ordered sets; the morphisms are the \emph{isotone maps}, that is, those maps $f\colon P\to Q$ such that $x\leq y$ implies $f(x)\leq f(y)$ whenever $x,y\in P$).
\end{itemize}

$\Seq X$ denotes the set of all (possibly empty) finite sequences of elements of any set~$X$.
Denote by~$\bx\conc\by$ the concatenation of finite sequences~$\bx$ and~$\by$.
For a finite sequence $\bx=(x_1,\dots,x_n)$, we set $n=\lh(\bx)$, the \emph{length of~$\bx$} (thus equal to~$0$ for the empty sequence).
We will often use the notational convention
 \begin{equation}\label{Eq:NotFinSeq}
 \bx=(x_1,\dots,x_n)\,,\quad
 \text{for }\bx\in\Seq X\,,\text{ where }n=\lh(\bx)\,.
 \end{equation}
As in Rotman \cite[Chapter~11]{Rotm1995}, a \emph{simplicial complex} (or \emph{abstract simplicial complex})~$K$ is a collection of nonempty finite subsets, called \emph{simplices}, of a set~$\VR K$, the \emph{vertices} of~$K$, such that every~$\set{v}$, where $v\in\VR K$, is a simplex, and every nonempty subset of a simplex is a simplex.
The \emph{$n$-skeleton} of~$K$, denoted by~$K^{(n)}$, is the set of all simplices with at most~$n+1$ elements (also called \emph{$n$-simplices}), for each $n\in\NN$.
We say that~$K$ \emph{has dimension at most~$n$} if $K=K^{(n)}$.
A simplicial complex~$K'$ is a \emph{subcomplex} of~$K$ if every simplex of~$K'$ is a simplex of~$K$.
A subcomplex~$K'$ of~$K$ is \emph{spanning} if it has the same vertices as~$K$.

A \emph{path}, in~$K$, is defined as a nonempty finite sequence $\bx=(x_0,\dots,x_n)$, where each $x_i\in\VR K$ and each $\set{x_i,x_{i+1}}\in K^{(1)}$.
Set $\sr{\bx}\eqdef x_0$ and $\tg{\bx}\eqdef x_n$, and say that~\emph{$\bx$ is a path from~$x_0$ to~$x_n$}.
The \emph{homotopy} relation, on the set $\Path{K}$ of all paths of~$K$, is the equivalence relation~$\simeq$ on $\Path{K}$ generated by all pairs
 \begin{align*}
 \bu\conc(x,x)\conc\bv&\simeq\bu\conc(x)\conc\bv\,,
 &&\text{where }\bu\conc(x)\conc\bv\text{ is a path},\\
 \bu\conc(x,y,z)\conc\bv&\simeq
 \bu\conc(x,z)\conc\bv\,,
 &&\text{where }\bu\conc(x,y,z)\conc\bv\text{ is a path}\\
 &&&\text{and }\set{x,y,z}\text{ is a simplex}\,.
 \end{align*}
Paths $\bx=\bu\conc(z)$ and $\by=(z)\conc\bv$ can be multiplied, by setting $\bx\by=\bu\conc(z)\conc\bv$ (\emph{not to be confused with the concatenation $\bx\conc\by$}), and this partial operation defines a category structure on $\Path{K}$, whose identities are the one-entry paths of~$K$.
We say that~$K$ is \emph{connected} if there is a path between any two vertices of~$K$.
Observe that $\bx\simeq\by$ implies that $\sr{\bx}=\sr{\by}$ and $\tg{\bx}=\tg{\by}$, for all paths~$\bx$ and~$\by$.
Denoting by $[\bx]=[x_0,\dots,x_n]$, or $[\bx]_K=[x_0,\dots,x_n]_K$ in case~$K$ needs to be specified, the homotopy class of a path $\bx=(x_0,\dots,x_n)$, homotopy classes can be multiplied, \emph{via} the rule $[\bx]\cdot[\by]=[\bx\by]$, which is defined if{f} $\tg{\bx}=\sr{\by}$.
The collection~$\hg{K}$ of all homotopy classes of paths of~$K$ is a \emph{groupoid}, called the \emph{edge-path groupoid}, or \emph{fundamental groupoid}, of~$K$.
The inverse $[\bx]^{-1}$ of a homotopy class~$[\bx]$, where $\bx=(x_0,\dots,x_n)$, is the homotopy class of $\bx^{-1}\eqdef(x_n,\dots,x_0)$.

For any poset~$P$, the \emph{chain complex} $\Sim{P}$ of~$P$ has vertices the elements of~$P$, and simplices the finite chains of~$P$.
An element~$x$ of~$P$ is a \emph{lower cover} of an element~$y$ of~$P$, in notation $x\prec y$, if $x<y$ and there is no~$z$ such that $x<z<y$.

For any simplicial complex~$K$, the \emph{barycentric subdivision} of~$K$ is a poset, defined as the set of all simplices of~$K$, partially ordered under set inclusion.

The \emph{height} of a poset~$P$ is defined as the supremum of the cardinalities of all chains of~$P$, minus one.

We say that a group~$G$ is a \emph{free factor} (resp., \emph{doubly free factor}) of a group~$H$, if there exists a group (resp., a free group)~$F$ such that $H\cong F* G$, where~$*$ denotes the free product (i.e., coproduct) in the Category~$\GP$ of all groups.
 
We denote by~$\Fm{X}$ (resp., $\Fg{X}$) the free monoid (resp., group) on a set~$X$.

We denote by $\NN=\set{0,1,2,\dots}$ the additive monoid of all nonnegative integers, and by~$\ZZ$ the additive group of all integers.

\section{The universal monoid of a category}\label{S:UnivMon}

The universal monoid construction, applied to a category~$S$, is a special case of a construction described in Chapters~8 to~10 of Higgins~\cite{Higg1971}.
Its underlying monoid is obtained by keeping all existing products in~$S$ and collapsing all the identities of~$S$.
Although most of the material in this section is contained in some form in Higgins~\cite{Higg1971}, we will write it in some detail, in order to be able to apply it to the category-to-monoid transfer results of Sections~\ref{S:UmScanc} and \ref{S:divetgcd}.

For a partial binary operation~$\cdot$ on a set~$S$, we will abbreviate the statement that~$x\cdot y$ is defined (resp., undefined) by writing~$x\cdot y\downarrow$ (resp., $x\cdot y\uparrow$), for any $x,y\in S$.
Hence, $x\cdot y\downarrow$ is equivalent to the statement $(\exists z)(z=x\cdot y)$.
This definition is extended the usual way to arbitrary \emph{terms}, of the language of semigroups, with parameters in a given semigroup~$S$: for example, for any $x,y,z\in S$, $(x\cdot y)\cdot z\downarrow$ holds if there are $u,v\in S$ such that $u=x\cdot y$ and $v=u\cdot z$; and $t\uparrow$ means that the term~$t$ is undefined (in~$S$).

\begin{definition}\label{D:Semicat}
A \emph{semicategory} is a structure $(S,\cdot)$, consisting of a set~$S$ endowed with a partial binary operation~$\cdot$ such that
$x\cdot(y\cdot z)\downarrow$ if{f} $(x\cdot y)\cdot z\downarrow$ if{f} $x\cdot y\downarrow$ and $y\cdot z\downarrow$, and then $x\cdot(y\cdot z)=(x\cdot y)\cdot z$ (sometimes denoted by~$x\cdot y\cdot z$, and usually by~$xyz$), for all $x,y,z\in S$.

An element~$e\in S$ is an \emph{identity} of~$S$ if it is idempotent (i.e., $e^2=e$) and $xe\downarrow$ implies that $xe=x$ and $ex\downarrow$ implies that $ex=x$, for all $x\in S$.
We denote by~$\Idt S$ the set of all identities of~$
S$.

A \emph{category} is a semicategory in which for every~$x$ (thought of as an arrow) there are (necessarily unique) identities~$a$ and~$b$ such that $x=ax=xb$.
We will write $a=\sr{x}$ (the \emph{source} of~$x$) and $b=\tg{x}$ (the \emph{target} of~$x$).
The hom-sets of~$S$ are then the
$S(a,b)=\setm{x\in S}{\sr{x}=a\text{ and }\tg{x}=b}$, for $a,b\in\Idt S$.

For categories~$S$ and~$T$, a map $f\colon S\to T$ is a \emph{functor} if $f(xy)=f(x)f(y)$ whenever $xy$ is defined, and~$f$ sends identities to identities.
In particular, if~$T$ is a \emph{monoid} (i.e., a category with exactly one identity, then denoted by~$1$), $f$ should send every identity to~$1$.
\end{definition}

Of course, if a category~$S$ is described, in the usual fashion, by its objects and morphisms, the identities of~$S$ are exactly the identity morphisms on its objects.
We will sometimes write $\sr{a}\stackrel{a}{\to}\tg{a}$, whenever $a\in S$.

\goodbreak

\begin{definition}\label{D:InvCancetc}
Let~$S$ be a category.
An element $a\in S$ is 
\begin{itemize}
\item\emph{right invertible} if there exists $x\in S$ (which is then called a \emph{right inverse} of~$a$) such that~$ax$ is an identity --- then of course, necessarily, $ax=\sr{a}$;

\item\emph{left invertible} if there exists $x\in S$ (which is then called a \emph{left inverse} of~$a$) such that~$xa$ is an identity --- then of course, necessarily, $xa=\tg{a}$;

\item\emph{invertible} if it is both left and right invertible;

\item\emph{left cancellable} (or \emph{monic}), if $ax=ay$ implies that $x=y$, for all $x,y\in S$;

\item\emph{right cancellable} (or \emph{epic}), if $xa=ya$ implies that $x=y$, for all $x,y\in S$;

\item\emph{cancellable} if it is both right and left cancellable.
\end{itemize}
The category~$S$ is \emph{left cancellative} (\emph{right cancellative}, \emph{cancellative}, respectively) if every element of~$S$ is left cancellable (right cancellable, cancellable, respectively).
A \emph{groupoid} is a category in which every element is invertible.
\end{definition}

{}From now on until Proposition~\ref{P:UgSEUniv} we shall fix a category~$S$.
For~$\ba,\bb\in\Seq S$, we say that~$\ba$ \emph{reduces to~$\bb$ in one step}, in notation $\ba\rightarrow\bb$, if there are $\bu,\bv\in\Seq S$ such that
either there is $e\in\Idt S$ such that $\ba=\bu\conc(e)\conc\bv$ and $\bb=\bu\conc\bv$, or there are $x,y\in S$ such that $xy\downarrow$\,, $\ba=\bu\conc(x,y)\conc\bv$, and $\bb=\bu\conc(xy)\conc\bv$.

Obviously, $\ba\rightarrow\bb$ implies that $\lh(\ba)=\lh(\bb)+1$.

We denote by~$\rightarrow^*$ the reflexive and transitive closure of~$\rightarrow$, and we say that~$\ba$ \emph{reduces to~$\bb$} if $\ba\rightarrow^*\bb$.
For all $\bx,\by\in\Seq S$, let $\bx\equiv\by$ hold if there is $\bz\in\Seq S$ such that $\bx\rightarrow^*\bz$ and $\by\rightarrow^*\bz$.

\goodbreak

\begin{lemma}\label{L:Confluent}
The following statements hold:
\begin{enumerater}
\item\label{arrconfl}
The union of the binary relation~$\rightarrow$, with the equality, is confluent: that is, whenever $\ba\rightarrow\bb_i$, for each $i\in\set{0,1}$, then either~$\bb_0=\bb_1$ or there exists $\bc\in\Seq S$ such that $\bb_0\rightarrow\bc$ and $\bb_1\rightarrow\bc$.

\item\label{arr*confl}
The binary relation~$\rightarrow^*$ is also confluent.

\item\label{arrcomp}
Both relations~$\rightarrow$ and~$\rightarrow^*$ are  compatible with concatenation, that is, $\bx\rightarrow\by$ \pup{resp., $\bx\rightarrow^*\by$} implies that $\bx\conc\bz\rightarrow\by\conc\bz$ and $\bz\conc\bx\rightarrow\bz\conc\by$ \pup{resp., $\bx\conc\bz\rightarrow^*\by\conc\bz$ and $\bz\conc\bx\rightarrow^*\bz\conc\by$}, for all $\bx,\by,\bz\in\Seq S$.

\item\label{EquivCong}
The binary relation~$\equiv$ is a monoid congruence on $(\Seq S,\conc)$.
\end{enumerater}
\end{lemma}

\begin{proof}
\emph{Ad}~\eqref{arrconfl}.
The problem reduces to a small number of cases, all of which are easy, and the only three of which not being completely trivial (e.g., requiring both definitions of a semicategory and an identity) being the following:
 \begin{multline*}
 \ba=\bu\conc(x,e)\conc\bv\,,\ 
 \bb_0=\bu\conc(x)\conc\bv\,,\ 
 \bb_1=\bu\conc(xe)\conc\bv\,,\\
 \text{where }\bu,\bv\in\Seq S\,,\ 
 x\in S\,,\ e\in\Idt S\,,\text{ and }xe\downarrow\,,
 \end{multline*}
\begin{multline*}
 \ba=\bu\conc(e,x)\conc\bv\,,\ 
 \bb_0=\bu\conc(x)\conc\bv\,,\ 
 \bb_1=\bu\conc(ex)\conc\bv\,,\\
 \text{where }\bu,\bv\in\Seq S\,,\ 
 x\in S\,,\ e\in\Idt S\,,\text{ and }ex\downarrow\,,
 \end{multline*}
and
\begin{multline*}
 \ba=\bu\conc(x,y,z)\conc\bv\,,\ 
 \bb_0=\bu\conc(x,yz)\conc\bv\,,\ 
 \bb_1=\bu\conc(xy,z)\conc\bv\,,\\
 \text{where }\bu,\bv\in\Seq S\,,\ 
 x,y,z\in S\,,\ xy\downarrow\,,\text{ and }
 yz\downarrow\,.
 \end{multline*}
In the first two cases, $\bb_0=\bb_1$.
In the third case, take $\bc=\bu\conc(xyz)\conc\bv$.

It is well known that~\eqref{arr*confl} follows from~\eqref{arrconfl}.

\eqref{arrcomp} is straightforward, and~\eqref{EquivCong} is an easy consequence of~\eqref{arr*confl} and~\eqref{arrcomp}.
\end{proof}

{}From now on we shall set, following the convention set in~\eqref{Eq:NotFinSeq},
 \[
 \Seqr S\eqdef\setm{\bx\in\Seq(S\setminus\Idt S)}
 {x_ix_{i+1}\uparrow
 \text{ whenever }1\leq i<i+1\leq\lh(\bx)}\,,
 \]
the set of all (finite) \emph{reduced sequences} of elements of~$S$.
Obviously, a finite sequence $\bx\in\Seq S$ is reduced if{f} there is no $\by\in\Seq S$ such that $\bx\rightarrow\by$, if{f}~$\bx$ is maximal with respect to the partial ordering~$\rightarrow^*$.
Observe that larger words, with respect to~$\rightarrow^*$, have smaller length.

\begin{lemma}\label{L:x2Ered}
For each $\bx\in\Seq S$, there exists a unique element of~$\Seqr S$, which we shall denote by~$\red{\bx}$, such that $\bx\rightarrow^*\red{\bx}$.
Furthermore, $\red{\bx}$ can be characterized in each of the following three ways:
\begin{itemize}
\item $\red{\bx}$ is the largest element of~$\bx/{\equiv}$ with respect to the partial ordering~$\rightarrow^*$;

\item $\red{\bx}$ is the unique element of~$\bx/{\equiv}$ with smallest length;

\item $\red{\bx}$ is the unique reduced sequence equivalent to~$\bx$ modulo~$\equiv$.
\end{itemize}
\end{lemma}

\begin{proof}
The existence statement follows trivially from the fact that~$\rightarrow^*$ decreases the length.
The uniqueness statement follows trivially from Lemma~\ref{L:Confluent}.
Then the proof of the equivalence of the three statements about~$\red{\bx}$ is straightforward.
\end{proof}

The last observation of Lemma~\ref{L:x2Ered} is, essentially, contained (with a different proof) in Theorem~4 of Chapter~10 in Higgins~\cite{Higg1971} (take $A=S$ and define~$\gs$ as the constant map with value~$1$).

\begin{notation}\label{Not:Umon}
We denote by~$\Um{S}$ the quotient monoid $(\Seq S)/{\equiv}$.
Moreover, we denote by~$\bx/{\equiv}$ the~$\equiv$-equivalence class of a finite sequence~$\bx\in\Seq S$, and we set $\ell(\bx)\eqdef\lh\pI{\red{\bx}}$.
\end{notation}

The elements of~$\Um{S}$ are exactly the equivalence classes~$\bx/{\equiv}$, where $\bx\in\Seq S$.
In particular, the unit element of~$\Um{S}$, which we shall denote by~$1$, is the $\equiv$-equivalence class of the empty sequence~$\es$.

Since every equivalence class $\bx/{\equiv}$ is uniquely determined by~$\red{\bx}$, which is its representative of smallest length (cf. Lemma~\ref{L:x2Ered}), the monoid~$\Um{S}$ can be alternatively described as consisting of all the reduced finite sequences (i.e., $\Seqr S$), endowed with the multiplication defined by
 \begin{equation}\label{Eq:MultSeqES}
 \bx\cdot_S\by=\red{(\bx\conc\by)}\,,
 \text{ for all }\bx,\by\in\Seqr S\,.
 \end{equation}

\begin{definition}\label{D:Conical}
A category~$S$ is \emph{conical} if $xy\in\Idt S$ (i.e., $xy$ is defined and belongs to~$\Idt S$) implies that~$x\in\Idt S$ (equivalently, $y\in\Idt S$), for any $x,y\in S$.
(Of course, in that case, $x=y=xy$.)
\end{definition}

The following result shows that the description of the multiplication given in~\eqref{Eq:MultSeqES} is especially convenient in case~$S$ is conical.

\begin{lemma}\label{L:Mult4conical}
Let the category~$S$ be conical.
Then the multiplication given by~\textup{\eqref{Eq:MultSeqES}} is characterized by
$1\cdot_S\bx=\bx\cdot_S1=\bx$, for any $\bx\in\Um{S}$, together with
 \begin{equation}\label{Eq:MultSeqEScon}
 \pI{\bu\conc(x)}\cdot_S\pI{(y)\conc\bv}=
 \begin{cases}
 \bu\conc(x,y)\conc\bv\,,&\text{if }xy\uparrow\\
 \bu\conc(xy)\conc\bv\,,&\text{if }xy\downarrow
 \end{cases},
 \text{ for all }\bu\conc(x),(y)\conc\bv\in\Seqr S\,.
 \end{equation}
\end{lemma}

\begin{proof}
Suppose that~$xy$ is defined.
Since~$S$ is conical and neither~$x$ nor~$y$ is an identity, $xy$ is not an identity, thus the finite sequence $\bu\conc(xy)\conc\bv$ remains reduced in the second case of~\eqref{Eq:MultSeqEScon}.
\end{proof}

In order to ease notation, we will usually write~$\bx\by$, or sometimes $\bx\cdot\by$ in case we need a separator, instead of~$\bx\cdot_S\by$.

A straightforward application of Lemma~\ref{L:Mult4conical} yields the following corollary.

\begin{corollary}\label{C:Ineq4ell}
The following statements hold, for any category~$S$:
\begin{enumerater}
\item $\ell(\bx\by)\leq\ell(\bx)+\ell(\by)$, for all $\bx,\by\in\Um{S}$.

\item
If, in addition, $S$ is conical, then
\begin{enumeratei}
\item $\ell(\bx)\leq\ell(\bx\by)$ and $\ell(\by)\leq\ell(\bx\by)$;

\item $\ell(\bx\by)\in\set{\ell(\bx)+\ell(\by)-1,\ell(\bx)+\ell(\by)}$,
\end{enumeratei}
for all $\bx,\by\in\Seqr S$.
\end{enumerater}
\end{corollary}
 
We will usually work with the description of~$\Um{S}$ as $\Seqr{S}$ endowed with the multiplication defined by~\eqref{Eq:MultSeqES}.
A noticeable exception, involving the description of~$\Um{S}$ as $(\Seq S)/{\equiv}$\,, is the following straightforward universal characterization of $\Um{S}$.

\begin{proposition}\label{P:UmSEUniv}
The map $\eps_S\colon S\to\Um{S}$, $x\mapsto(x)/{\equiv}$ is a functor.
Furthermore, the pair $(\Um{S},\eps_S)$ is an initial object in the Category of all functors from~$S$ to a monoid.
Hence, $\Umn$ defines a left adjoint of the inclusion Functor $\MON\hookrightarrow\CAT$.
\end{proposition}


Accordingly, we shall call the monoid~$\Um{S}$ (or, more precisely, the pair\linebreak $(\Um{S},\eps_S)$) the \emph{universal monoid} of~$S$.

The requirement that~$\eps_S$ be a functor --- thus collapses all identities --- implies that~$\eps_S$ is not one-to-one as a rule (so~$S$ may not be a partial subsemigroup of~$\Um{S}$).
Nevertheless, it falls short of being so.
The following observation is a specialization, to the construction~$\Um{S}$, of Corollary~1 of Chapter~10 in Higgins~\cite{Higg1971}.

\begin{lemma}\label{L:KerofepsS}
For all $x,y\in S$, $\eps_S(x)=\eps_S(y)$ if{f} either $x=y$ or~$x$ and~$y$ are both identities.
In particular, the restriction of~$\eps_S$ to every hom-set of~$S$ is one-to-one.
\end{lemma}

\begin{proof}
It is obvious that~$\red{(x)}$ is equal to~$(x)$ if~$x$ is not an identity, and the empty sequence otherwise.
The first statement follows immediately.

Now let $a,b\in\Idt S$ and let $x,y\in S(a,b)$ such that $\eps_S(x)=\eps_S(y)$.
If $x\neq y$, then, by the paragraph above, $x=y=a$, a contradiction.
\end{proof}

\begin{notation}\label{Not:Ugrp}
We denote by~$\Ug{M}$ the universal group of a monoid~$M$ --- that is, the initial object in the Category of all monoid homomorphisms from~$M$ to a group.
We extend this notation to all categories, by setting
 \[
 \Ug{S}\eqdef\Ug{\Um{S}}\,
 \]
for every category~$S$.
\end{notation}

Pre-composing the canonical monoid homomorphism $\Um{S}\to\Ug{S}$ with the canonical homomorphism $\eps_S\colon S\to\Um{S}$ yields the canonical homomorphism $\eta_S\colon S\to\Ug{S}$, and the following easy consequence of Proposition~\ref{P:UmSEUniv}.

\begin{proposition}\label{P:UgSEUniv}
The map $\eta_S\colon S\to\Ug{S}$ is a group-valued functor.
Furthermore, the pair $(\Ug{S},\eta_S)$ is an initial object in the Category of all functors from~$S$ to a group.
Hence, $\Ugn$ defines a left adjoint of the inclusion Functor $\GP\hookrightarrow\CAT$.
\end{proposition}

Accordingly, we shall call the group~$\Ug{S}$ (or, more precisely, the pair\linebreak $(\Ug{S},\eta_S)$) the \emph{universal group} of~$S$.

\begin{remark}\label{Rk:Cat2SE}
For any category~$S$, every element $\bx\in\Um{S}$ can be written in a unique way as a product $\eps_S(x_1)\cdots\eps_S(x_n)$, which we will often abbreviate as $x_1\cdots x_n$, where no~$x_i$ is an identity and no product~$x_ix_{i+1}$ is defined: namely, let $\red{\bx}=(x_1,\dots,x_n)$.
The elements of $S\setminus\Idt S$, here identified with their images under~$\eps_S$, will be called the \emph{standard generators} of the monoid~$\Um{S}$.

In particular, if~$S$ is a groupoid, then every element of~$\Um{S}$ is a product of invertibles, thus is itself invertible; and thus $\Um{S}=\Ug{S}$ is the universal group of~$S$.
\end{remark}

\section{Preservation of conicality and cancellativity}
\label{S:UmScanc}

The main goal of this section is to prove that conicality and cancellativity of a category can be read on its universal monoid.
The case of conicality is easy.

\begin{proposition}\label{P:PresConical}
A category~$S$ is conical if{f} its universal monoid~$\Um{S}$ is conical.
\end{proposition}

\begin{proof}
Let~$\Um{S}$ be conical and let $a,b\in S$ such that $ab$ is an identity.
Then $\eps_S(a)\eps_S(b)=\eps_S(ab)=1$, thus, as~$\Um{S}$ is conical, $\eps_S(a)=\eps_S(b)=1$, and thus~$a$ and~$b$ are both identities.

Observing that $\ell(\bx)=0$ if{f} $\bx=1$, for any $\bx\in\Um{S}$, it follows immediately from Corollary~\ref{C:Ineq4ell} that if~$S$ is conical, then so is~$\Um{S}$.
\end{proof}

Preservation of cancellativity and invertibility are slightly less straightforward.

\begin{lemma}\label{L:PresCancella}
The following statements hold, for any element~$a$ in a category~$S$:
\begin{enumerater}
\item\label{leftcanc}
$a$ is left cancellable in~$S$ if{f}~$\eps_S(a)$ is left cancellable in~$\Um{S}$;

\item\label{rightcanc}
$a$ is right cancellable in~$S$ if{f}~$\eps_S(a)$ is right cancellable in~$\Um{S}$;

\item\label{rightinv}
$a$ is right invertible in~$S$ if{f} $\eps_S(a)$ is right invertible in~$\Um{S}$;

\item\label{leftinv}
$a$ is left invertible in~$S$ if{f} $\eps_S(a)$ is left invertible in~$\Um{S}$.
\end{enumerater}
\end{lemma}

\begin{proof}
We prove the results about left cancellativity and right invertibility.
The results about right cancellativity and left invertibility follow by symmetry.

Let~$\eps_S(a)$ be left cancellable in~$\Um{S}$ and let $x,y\in S$ such that $ax=ay$.
In particular, $x$ and~$y$ have the same source and the same target.
Moreover, since $\eps_S(a)\eps_S(x)=\eps_S(a)\eps_S(y)$ and $\eps_S(a)$ is left cancellable in~$\Um{S}$, we get $\eps_S(x)=\eps_S(y)$. 
By the second part of Lemma~\ref{L:KerofepsS}, it follows that $x=y$.

Suppose, conversely, that~$a$ is left cancellable in~$S$.
We need to prove that~$\eps_S(a)$ is left cancellable in~$\Um{S}$.

Suppose first that~$a$ is right invertible in~$S$ and let~$s$ be a right inverse of~$a$ in~$S$.
{}From $as=\sr{a}$ it follows that $asa=a=a\tg{a}$, thus, since~$a$ is left cancellable in~$S$, $sa=\tg{a}$.
It follows that~$a$ is invertible in~$S$, so~$\eps_S(a)$ is invertible in~$\Um{S}$ and the desired conclusion holds.

Suppose from now on that~$a$ is not right invertible in~$S$.
For every reduced sequence $\bz=(z_1,\dots,z_k)\in\Um{S}$, two cases can occur:
\begin{itemize}
\item either~$k=0$ or~$az_1$ is undefined.
Then $\eps_S(a)\bz=(a,z_1,\dots,z_k)$.

\item $k>0$ and $az_1$ is defined.
Then $\eps_S(a)\bz=(az_1,z_2,\dots,z_k)$.
\end{itemize}
The putative ``third case'' suggested by the above, namely $k>0$ and $az_1$ is an identity of~$S$, does not occur, because~$a$ is not right invertible in~$S$.
In particular, observe the following:
 \begin{equation}\label{Eq:l(az)}
 \ell\pI{\eps_S(a)\bz}\in\set{\ell(\bz),\ell(\bz)+1}\,.
 \end{equation}
Now let $\bx=(x_1,\dots,x_m)$ and $\by=(y_1,\dots,y_n)$ be elements in~$\Um{S}$ (reduced sequences) such that $\eps_S(a)\bx=\eps_S(a)\by$; we must prove that $\bx=\by$.
We may assume that $m\leq n$.
It follows from~\eqref{Eq:l(az)} that either $n=m$ or $n=m+1$.
Since the case where $m=n=0$ is trivial, there are three cases left to consider.

\subsection*{Case~1}
$m=n>0$ and $ax_1$, $ay_1$ are both undefined.
Then
 \[
 (a,x_1,\dots,x_m)=\eps_S(a)\bx=\eps_S(a)\by
 =(a,y_1,\dots,y_m)\,,
 \]
thus $x_i=y_i$ whenever $1\leq i\leq m$.

\subsection*{Case~2}
$m=n>0$ and $ax_1$, $ay_1$ are both defined.
Then
 \[
 (ax_1,x_2,\dots,x_m)=\eps_S(a)\bx=\eps_S(a)\by
 =(ay_1,y_2,\dots,y_m)\,,
 \]
thus $ax_1=ay_1$ and $x_i=y_i$ whenever $2\leq i\leq m$.
Furthermore, from the first equation and the left cancellativity of~$a$ it follows that $x_1=y_1$.

\subsection*{Case 3}
$n=m+1$, $ax_1$ is undefined, and~$ay_1$ is defined.
Then
 \[
 (a,x_1,\dots,x_m)=\eps_S(a)\bx=\eps_S(a)\by
 =(ay_1,y_2,\dots,y_{m+1})\,,
 \]
thus $ay_1=a=a\tg{a}$, and thus $y_1=\tg{a}$, an identity of~$S$; in contradiction with the sequence~$\by$ being reduced.

Let us now suppose that~$a$ is right invertible in~$S$, that is, $as=\sr{a}$ for some $s\in S$.
Then $\eps_S(a)\eps_S(s)=1$, thus~$\eps_S(a)$ is right invertible in~$\Um{S}$.

Suppose, finally, that~$\eps_S(a)$ is right invertible in~$\Um{S}$, we must prove that~$a$ is right invertible in~$S$.
Suppose otherwise.
By assumption, there is a reduced sequence $\bb=(b_1,\dots,b_n)\in\Um{S}$ such that $\eps_S(a)\bb=1$.
In particular, $(a,b_1,\dots,b_n)$ is reducible, thus, since~$a$ is not an identity and~$\bb$ is a reduced sequence, $n>0$ and $ab_1$ is defined.
Since $(ab_1,b_2,\dots,b_n)$ should be equivalent to the empty sequence, it is reducible, thus $ab_1$ is an identity, a contradiction.
\end{proof}

Since every element of~$\Um{S}$ is a finite product of elements of the range of~$\eps_S$, and since left cancellativity, right cancellativity, and invertibility are all preserved under finite products, we obtain immediately the following.

\begin{corollary}\label{C:PresCancella}
A category~$S$ is left cancellative \pup{resp., right cancellative} if{f} its enveloping monoid~$\Um{S}$ is left cancellative \pup{resp., right cancellative}.
Moreover, $S$ is a groupoid if{f}~$\Um{S}$ is a group.
\end{corollary}

\section{Gcd-categories and gcd-monoids}
\label{S:divetgcd}

In this section we extend, to an arbitrary category, Dehornoy's definition of a gcd-monoid, and we prove, in the same spirit as in Section~\ref{S:UmScanc}, that a category is a gcd-category if{f} its universal monoid is a gcd-monoid.

\begin{definition}\label{D:diveS}
Every category~$S$ carries two partial preorderings~$\dive_S$ and~$\divet_S$, defined by letting $a\dive_Sb$ (resp., $a\divet_Sb$) hold if there exists $x\in S$ such that $b=ax$ (resp., $b=xa$) --- we say that~$a$ \emph{left divides} (resp., \emph{right divides})~$b$, or, equivalently, that~$b$ is a \emph{right multiple} (resp., \emph{left multiple}) of~$a$.
If such an~$x$ is unique, then we write $x=b\cp_Sa$ (resp., $x=b\cpt_Sa$).
We will call~$\dive_S$ (resp., $\divet_S$) the \emph{left divisibility preordering} (resp., the \emph{right divisibility preordering}) of~$S$.

Following tradition, the greatest lower bound of a subset~$X$ of~$S$ if it exists, with respect to~$\dive_S$ (resp., $\divet_S$), will be called the \emph{left gcd} (resp., \emph{right gcd}) of~$X$ and denoted by $\Gcd_SX$ (resp., $\Gcdt_SX$).
Likewise, the least upper bound of a subset~$X$ of~$S$ if it exists, with respect to~$\dive_S$ (resp., $\divet_S$), will be called the \emph{right lcm} (resp., \emph{left lcm}) of~$X$ and denoted by $\Lcm_SX$ (resp., $\Lcmt_SX$).
If $X=\set{a,b}$, we will write $a\gcd_Sb$, $a\gcdt_Sb$, $a\lcm_Sb$, $a\lcmt_Sb$, respectively.

As usual, we drop the index~$S$ in the notations above in case~$S$ is understood.
\end{definition}

Observe, in particular, that an element~$a\in S$ is right invertible (resp., left invertible) if{f} $a\dive_S\sr{a}$ (resp., $a\divet_S\tg{a}$).

\begin{definition}\label{D:Delta01}
For a reduced sequence $\bx=(x_1,\dots,x_n)\in\Um{S}$, with $n>0$, the \emph{left component}~$\SR{\bx}$ and the \emph{right component}~$\TG{\bx}$ of~$\bx$ are the elements of~$S$ respectively defined as $\SR{\bx}=x_1$ and $\TG{\bx}=x_n$.
\end{definition}

Although it is tempting to extend Definition~\ref{D:Delta01} by setting $\SR{1}=\TG{1}=\one$, a new ``unit'' element, this element~$\one$ would need to live outside the category~$S$, and the new structure $S\cup\set{\one}$ would usually no longer be a category, which could open the way to a whole collection of incidents.
We therefore chose to restrict the domain of definition of both~$\SRn$ and~$\TGn$ to $\Um{S}\setminus\set{1}$.

The following result states that the component maps~$\SRn$ and~$\TGn$ are right adjoints of the canonical map $\eps_S\colon S\to\Um{S}$ with respect to the divisibility preorderings on~$S$ and~$\Um{S}$.

\begin{lemma}\label{L:AdjepsDelta}
Let~$S$ be a category and let $a\in S$.
Set $M=\Um{S}$.
Then $\eps_S(a)\dive_M1$ if{f} $a\dive_S\sr{a}$ and $\eps_S(a)\divet_M1$ if{f} $a\divet_S\tg{a}$.
Furthermore,
 \begin{align}
 \text{if }a\text{ is not right invertible, then}\qquad
 \eps_S(a)\dive_M\bb\quad&\text{if{f}}\quad
 a\dive_S\SR{\bb}\,,\label{Eq:AdjepsDelta}\\
 \text{if }a\text{ is not left invertible, then}\qquad
 \eps_S(a)\divet_M\bb\quad&\text{if{f}}\quad
 a\divet_S\TG{\bb}\,,\label{Eq:AdjepsDeltat} 
 \end{align}
for any $\bb\in M\setminus\set{1}$.
\end{lemma}

\begin{proof}
It is sufficient to prove the statements about~$\dive$.
The equivalence between $\eps_S(a)\dive_M1$ and $a\dive_S\sr{a}$ follows from Lemma~\ref{L:PresCancella}\eqref{rightinv}.
Now let $a\ndive_S\sr{a}$ and $\bb\in M\setminus\set{1}$, we need to prove that $\eps_S(a)\dive_M\bb$ if{f} $a\dive_S\SR{\bb}$.
The implication from the right to the left is obvious since then, $\eps_S(a)\dive_M\eps_S(\SR{\bb})\dive_M\bb$.
Suppose, conversely, that $\eps_S(a)\dive_M\bb$, we need to prove that $a\dive_S\SR{\bb}$.
Let $\bc=(c_1,\dots,c_n)\in M$ such that $\bb=\eps_S(a)\bc$.
Since~$a$ is not right invertible, $ac_1$ is not an identity (if $n>0$ and $ac_1\downarrow$), thus we get
 \[
 \bb=\eps_S(a)\bc=\begin{cases}
 (a,c_1,\dots,c_n)\,,&\text{if }n=0\text{ or }
 ac_1\uparrow\,,\\
 (ac_1,c_2,\dots,c_n)\,,&\text{if }n>0\text{ and }
 ac_1\downarrow\,.
 \end{cases}
 \]
In the first case, $\SR{\bb}=a$.
In the second case, $\SR{\bb}=ac_1$.
In both cases, we get $a\dive_S\SR{\bb}$.
\end{proof}

For an interpretation of Lemma~\ref{L:AdjepsDelta} in terms of normal forms and Garside families, see Remark~\ref{Rk:AdjepsDelta}.

\begin{lemma}\label{L:Cancell}
Let~$S$ be a conical category.
If~$S$ is left \pup{resp., right} cancellative, then the left \pup{resp., right} divisibility preordering of~$S$ is a partial ordering.
\end{lemma}

\begin{proof}
We prove the statement about the left divisibility preordering~$\dive_S$ of~$S$.
Let $a\dive_Sb\dive_Sa$, where $a,b\in S$.
There are $x,y\in S$ such that $b=ax$ and $a=by$; whence $a\tg{a}=a=axy$.
Since~$S$ is left cancellative, $\tg{a}=xy$.
Since~$S$ is conical, $x=y=\tg{a}$, so $a=b$.
\end{proof}

Although the conicality assumption cannot be weakened in Lemma~\ref{L:Cancell}, the cancellativity assumption can.
For example, defining~$S$ as the category with three objects~$e_0$, $e_1$, $e_2$ and arrows~$a$, $b$ with $e_0\stackrel{a,b}{\to}e_1$ and $e_1\stackrel{c}{\to}e_2$ with $ac=bc$, both left and right divisibility preorderings of~$S$ are orderings, while~$S$ is not left cancellative.

Say that a finite sequence $\ba=(a_1,\dots,a_m)$ is a \emph{prefix} (resp., \emph{suffix}) of a finite sequence $\bb=(b_1,\dots,b_n)$ if $m\leq n$ and $a_i=b_i$ (resp., $a_i=b_{n-m+i}$) whenever $1\leq i\leq m$.
It is trivial that $\dive_{\Um{S}}$ (resp., $\divet_{\Um{S}}$) extends the prefix ordering (resp., suffix ordering) on~$\Um{S}$.
As the following lemma shows, more is true.

%
%
%

\begin{lemma}\label{L:DivRedSeq}
Let~$S$ be a conical category and set $M=\Um{S}$.
The following statements hold, for all $\ba,\bb\in\Um{S}$:
\begin{enumerater}
\item\label{a1mdiveb1n}
$\ba\dive_M\bb$ if{f} either~$\ba$ is a prefix of~$\bb$ or there are $\ba',\bb'\in\Um{S}$ and $u,v\in S$ such that $\ba=\ba'\conc(u)$, $\bb=\ba'\conc(v)\conc\bb'$, and $u\dive_Sv$.

\item\label{a1mdivetb1n}
$\ba\divet_M\bb$ if{f} either~$\ba$ is a suffix of~$\bb$ or there are $\ba',\bb'\in\Um{S}$ and $u,v\in S$ such that $\ba=(u)\conc\ba'$, $\bb=\bb'\conc(v)\conc\ba'$, and $u\divet_Sv$.

\end{enumerater}
\end{lemma}

\begin{proof}
By symmetry, it is sufficient to prove~\eqref{a1mdiveb1n}.
Write $\ba=(a_1,\dots,a_m)$ and $\bb=(b_1,\dots,b_n)$.

Suppose that $\ba\dive_M\bb$ and let $\bx\in M$ such that $\bb=\ba\bx$.
We establish the condition stated in the right hand side of~\eqref{a1mdiveb1n}.
Write $\bx=(x_1,\dots,x_k)$ (a reduced sequence).
It follows from Corollary~\ref{C:Ineq4ell} that $m\leq n$.
Furthermore, by Lemma~\ref{L:Mult4conical}, there are two cases to consider.

\subsection*{Case 1}
$a_mx_1$ is undefined.
Then $\bb=\ba\conc\bx$, so $\ba$ is a prefix of~$\bb$ and the desired conclusion holds.

\subsection*{Case 2}
$a_mx_1$ is defined.
The desired conclusion holds with $\ba'=(a_1,\dots,a_{m-1})$, $\bb'=(x_2,\dots,x_k)$, $u=a_m$, and $v=a_mx_1$.

Following the proofs above in reverse yields easily the desired equivalence.
\end{proof}

The following technical lemma states in which case a left or right gcd exists in~$\Um{S}$, along with an algorithm enabling us to compute that gcd.

\begin{lemma}\label{L:meetUmS}
Let~$S$ be a conical category, set $M=\Um{S}$, let $\vecm{\ba_i}{i\in I}$ be a nonempty family of elements in~$\Um{S}$, and let $\bc\in M$.
We denote by~$\ba_{\mathrm{left}}$ \pup{resp., $\ba_{\mathrm{right}}$} the longest common prefix \pup{resp., suffix} of the~$\ba_i$.
Then the following statements hold:

\begin{enumerater}
\item\label{cdiveMallai}
$\bc\dive_M\ba_i$ for all~$i\in I$ if{f} either $\bc\dive_M\ba_{\mathrm{left}}$, or $\ba_{\mathrm{left}}\neq\ba_i$ for each~$i\in I$ and there is $c\in S$ such that $\bc=\ba_{\mathrm{left}}\conc(c)$ and $c\dive_S\SR{(\ba_i\cp_M\ba_{\mathrm{left}})}$ for each $i\in I$.

\item\label{cgcdMallai}
Suppose that~$S$ is left cancellative.
Then~$\bc$ is the left gcd of $\setm{\ba_i}{i\in I}$ in~$M$ if{f} one of the following cases occurs:
\begin{enumeratei}
\item\label{baleftcase}
$\bc=\ba_{\mathrm{left}}$, and either $\ba_{\mathrm{left}}=\ba_i$ for some~$i$, or $\setm{\sr{\SR{(\ba_i\cp_M\ba_{\mathrm{left}}})}}{i\in I}$ is not a singleton, or the left gcd of $\setm{\SR{(\ba_i\cp_M\ba_{\mathrm{left}})}}{i\in I}$ is an identity;

\item\label{baleft+case}
$\ba_i\neq\ba_{\mathrm{left}}$ for all~$i$, the left gcd~$c$ of $\setm{\SR{(\ba_i\cp_M\ba_{\mathrm{left}})}}{i\in I}$ exists in~$S$ and it is not an identity, and $\bc=\ba_{\mathrm{left}}\conc(c)$.
\end{enumeratei}

\item\label{cdivetMallai}
$\bc\divet_M\ba_i$ for all~$i\in I$ if{f} either $\bc\divet_M\ba_{\mathrm{right}}$, or $\ba_{\mathrm{right}}\neq\ba_i$ for each~$i\in I$ and there is $c\in S$ such that $\bc=(c)\conc\ba_{\mathrm{right}}$ and $c\divet_S\TG{(\ba_i\cpt_M\ba_{\mathrm{right}})}$ for each $i\in I$.

\item\label{cgcdtMallai}
Suppose that~$S$ is right cancellative.
Then~$\bc$ is the right gcd of $\setm{\ba_i}{i\in I}$ in~$M$ if{f} one of the following cases occurs:
\begin{enumeratei}
\item\label{barightcase}
$\bc=\ba_{\mathrm{right}}$ and either $\ba_{\mathrm{right}}=\ba_i$ for some~$i$ or $\setm{\tg{\TG{(\ba_i\cpt_M\ba_{\mathrm{right}}})}}{i\in I}$ is not a singleton, or the right gcd of $\setm{\TG{(\ba_i\cpt_M\ba_{\mathrm{right}})}}{i\in I}$ is an identity;

\item\label{baright+case}
$\ba_i\neq\ba_{\mathrm{right}}$ for all~$i$, the greatest lower bound~$c$ of $\setm{\TG{(\ba_i\cpt_M\ba_{\mathrm{right}})}}{i\in I}$ exists in~$S$ and it is not an identity, and $\bc=(c)\conc\ba_{\mathrm{right}}$.
\end{enumeratei}

\end{enumerater}
\end{lemma}

\begin{proof}
By symmetry, it suffices to prove~\eqref{cdiveMallai} and~\eqref{cgcdMallai}.
Since $\ba_{\mathrm{left}}\dive_M\ba_i$ for all~$i\in I$, $\bc\dive_M\ba_{\mathrm{left}}$ implies that $\bc\dive_M\ba_i$ for all $i\in I$.
Now suppose that $\ba\neq\ba_i$ for all~$i\in I$, and let $c\in S$ such that $\bc=\ba_{\mathrm{left}}\conc(c)$ and $c\dive_S\SR{(\ba_i\cp_M\ba_{\mathrm{left}})}$ for each $i\in I$.
Then
 \[
 \bc=\ba_{\mathrm{left}}\cdot\eps_S(c)\dive_M
 \ba_{\mathrm{left}}\cdot
 (\ba_i\cp_M\ba_{\mathrm{left}})
 =\ba_i\,,\quad\text{for all }i\in I\,.
 \]
Suppose, conversely, that $\bc\dive_M\ba_i$ for all $i\in I$.
We prove that the condition given on the right hand side of~\eqref{cdiveMallai} holds for~$\bc$.
We may assume that $\bc\neq1$, so $\bc=\bc_*\conc(c)$ for some $\bc_*\in M$ and some $c\in S\setminus\Idt S$.
Since~$S$ is conical and by Lemma~\ref{L:DivRedSeq}, $\ell(\bc)\leq\ell(\ba_i)$ for all $i\in I$ and~$\bc_*$ is a proper prefix of all~$\ba_i$, thus also of~$\ba_{\mathrm{left}}$.
Write $\ba_i=\bc_*\conc(a_{i,1},\dots,a_{i,n_i})$, for a reduced sequence $(a_{i,1},\dots,a_{i,n_i})$ with $n_i>0$.
Since $\bc=\bc_*\conc(c)\dive_M\ba_i$, it follows from Lemma~\ref{L:DivRedSeq} that
 \begin{equation}\label{Eq:cleqai1}
 c\dive_Sa_{i,1}\quad\text{for all }i\in I\,.
 \end{equation}
Suppose first that~$\bc_*$ is a proper prefix of $\ba_{\mathrm{left}}$.
Then the~$a_{i,1}$ are all equal to the entry of~$\ba_{\mathrm{left}}$ following~$\bc_*$, so they form a constant family, thus, by~\eqref{Eq:cleqai1} together with Lemma~\ref{L:DivRedSeq}, we get $\bc\dive_M\ba_{\mathrm{left}}$.

The only remaining possibility is $\bc_*=\ba_{\mathrm{left}}$, so (since each $n_i>0$) $\ba_{\mathrm{left}}$ is a proper prefix of each~$\ba_i$.
Since each $\ba_i=\ba_{\mathrm{left}}\conc(a_{i,1},\dots,a_{i,n_i})$, we get $a_{i,1}=\SR{(\ba_i\cp_M\ba_{\mathrm{left}})}$, and thus, by~\eqref{Eq:cleqai1}, $c\dive_S\SR{(\ba_i\cp_M\ba_{\mathrm{left}})}$.
This completes the proof of~\eqref{cdiveMallai}.

Now we prove~\eqref{cgcdMallai}.
The additional assumption that~$S$ be left cancellative, together with Proposition~\ref{P:PresConical} and Corollary~\ref{C:PresCancella}, ensures that~$M$ is also conical and left cancellative.
By Lemma~\ref{L:Cancell}, it follows that~$\dive_M$ is a partial ordering.

It follows from~\eqref{cdiveMallai} that the left gcd of $\setm{\ba_i}{i\in I}$, if it exists, is either equal to~$\ba_{\mathrm{left}}$ or has the form $\ba_{\mathrm{left}}\conc(c)$ for some~$c$.
The case where $\ba_{\mathrm{left}}=\ba_i$, for some~$i$, being trivial, we may assume that~$\ba_{\mathrm{left}}$ is a proper prefix of each~$\ba_i$.
In that case, by~\eqref{cdiveMallai}, the gcd of $\setm{\ba_i}{i\in I}$ is equal to~$\ba_{\mathrm{left}}$ if{f} either the set $\setm{\SR{(\ba_i\cp_M\ba_{\mathrm{left}})}}{i\in I}$ has no left divisor, that is,  $\setm{\sr{\SR{(\ba_i\cp_M\ba_{\mathrm{left}}})}}{i\in I}$ is not a singleton, or the only left divisor of the set $\setm{\SR{(\ba_i\cp_M\ba_{\mathrm{left}})}}{i\in I}$ is an identity.
Again by~\eqref{cdiveMallai}, the gcd of $\setm{\ba_i}{i\in I}$ exists and is not equal to~$\ba_{\mathrm{left}}$ if{f} it has the form $\ba_{\mathrm{left}}\conc(c)$ where~$c$ is the left gcd of $\setm{\SR{(\ba_i\cp_M\ba_{\mathrm{left}})}}{i\in I}$ in~$S$ and~$c$ is not an identity.
\end{proof}

A direct application of Lemma~\ref{L:meetUmS} to one-entry sequences yields the following.

\begin{corollary}\label{C:meetUmS}
The following statements hold, for any conical category~$S$ and any $a,b\in S$:
\begin{enumerater}
\item\label{agcdb}
If~$S$ is left cancellative and $\sr{a}=\sr{b}$, then $\eps_S(a)\gcd\eps_S(b)$ exists in~$\Um{S}$ if{f} $a\gcd b$ exists in~$S$.
Furthermore, if this holds, then $\eps_S(a\gcd b)=\eps_S(a)\gcd\eps_S(b)$.

\item\label{agcdtb}
If~$S$ is right cancellative and $\tg{a}=\tg{b}$, then $\eps_S(a)\gcdt\eps_S(b)$ exists in~$\Um{S}$ if{f} $a\gcdt b$ exists in~$S$.
Furthermore, if this holds, then $\eps_S(a\gcdt b)=\eps_S(a)\gcdt\eps_S(b)$.
\end{enumerater}
\end{corollary}

\begin{definition}\label{D:LeftgcdCat}
A category~$S$ is a \emph{left gcd-category} if it is conical, left cancellative, and any elements $a,b\in S$ such that $\sr{a}=\sr{b}$ have a left gcd.
\emph{Right gcd-categories} are defined dually.
We say that~$S$ is a \emph{gcd-category} if it is simultaneously a left and right gcd-category.

If~$S$ is a monoid, then we say that~$S$ is a \emph{left gcd-monoid}, \emph{right gcd-monoid}, \emph{gcd-monoid}, respectively.
\end{definition}

The terminology introduced in Definition~\ref{D:LeftgcdCat} is consistent with the one of Dehornoy~\cite{Dis} (see Section~\ref{Su:BackHigh} for more background on gcd-monoids).
Gcd-monoids and gcd-categories are related by the following result.

\begin{theorem}\label{T:UmSgcdmon}
A category~$S$ is a left gcd-category \pup{right gcd-category, gcd-category, respectively} if{f} its universal monoid $\Um{S}$ is a left gcd-monoid \pup{right gcd-monoid, gcd-monoid, respectively}.
\end{theorem}

\begin{proof}
It suffices to establish the statement for left gcd-categories and left gcd-monoids.
Set $M=\Um{S}$.

It follows from Proposition~\ref{P:PresConical} and Corollary~\ref{C:PresCancella} that~$S$ is conical and left cancellative if{f}~$M$ is conical and left cancellative.

Now suppose that~$M$ is a left gcd-monoid and let $a,b\in S$ such that $\sr{a}=\sr{b}$.
We must prove that $\set{a,b}$ has a left gcd in~$S$.
We may assume that $a\neq b$ and neither~$a$ nor~$b$ is an identity.
Hence~$(a)$ and~$(b)$ are both reduced sequences.
By assumption, they have a left gcd in~$M$.
Since they are distinct, their largest common prefix is the empty sequence, thus, as $\SR{(a)}=a$ and $\SR{(b)}=b$ and by Lemma~\ref{L:meetUmS}\eqref{cgcdMallai}, $a\gcd_Sb$ exists.

Let, conversely, $S$ be a left gcd-category and let $\vecm{\ba_i}{i\in I}$ be a nonempty finite family of elements of~$M$.
Denote by~$\ba_{\mathrm{left}}$ the largest common prefix of the~$\ba_i$.
We apply Lemma~\ref{L:meetUmS}\eqref{cgcdMallai}.
If either $\ba_i=\ba_{\mathrm{left}}$ for some~$i$, or $\setm{\sr{\SR{(\ba_i\cp_M\ba_{\mathrm{left}}})}}{i\in I}$ is not a singleton, or the only left divisor of $\setm{\SR{(\ba_i\cp_M\ba_{\mathrm{left}})}}{i\in I}$ is an identity, then the left gcd of the~$\ba_i$ is~$\ba_{\mathrm{left}}$.
Now suppose that each~$\ba_i$ properly extends~$\ba_{\mathrm{left}}$ and that $\setm{\SR{(\ba_i\cp_M\ba_{\mathrm{left}})}}{i\in I}$ has a non-identity left divisor.
By assumption and since~$I$ is finite nonempty, this means that $\setm{\SR{(\ba_i\cp_M\ba_{\mathrm{left}})}}{i\in I}$ has a non-identity left gcd, say~$c$, in~$S$.
Observe that the sequence $\ba_{\mathrm{left}}\conc(c)$ is reduced.
By Lemma~\ref{L:meetUmS}\eqref{cgcdMallai}, it follows that $\ba_{\mathrm{left}}\conc(c)$ is the left gcd of $\setm{\ba_i}{i\in I}$ in~$M$.
\end{proof}

The following result extends Corollary~\ref{C:meetUmS} to lcms.
Since we did not define left or right lcm-categories, our assumption (existence of both left and right gcds) needs to be a bit stronger (cf. Dehornoy \cite[Lemma~2.14]{Dis}).

\begin{proposition}\label{P:Sclslcm}
The following statements hold, for any gcd-category~$S$ and any $a,b\in S$:
\begin{enumerater}
\item\label{alcmb}
If $\sr{a}=\sr{b}$, then $\eps_S(a)\lcm\eps_S(b)$ exists in~$\Um{S}$ if{f}~$\eps_S(a)$ and~$\eps_S(b)$ have a common right multiple in~$\Um{S}$, if{f} $a\lcm b$ exists in~$S$, if{f}~$a$ and~$b$ have a common right multiple in~$S$.
Furthermore, if this holds, then $\eps_S(a\lcm b)=\eps_S(a)\lcm\eps_S(b)$.

\item\label{alcmtb}
If $\tg{a}=\tg{b}$, then $\eps_S(a)\lcmt\eps_S(b)$ exists in~$\Um{S}$ if{f}~$\eps_S(a)$ and~$\eps_S(b)$ have a common left multiple in~$\Um{S}$, if{f} $a\lcmt b$ exists in~$S$, if{f}~$a$ and~$b$ have a common left multiple in~$S$.
Furthermore, if this holds, then $\eps_S(a\lcmt b)=\eps_S(a)\lcmt\eps_S(b)$.
\end{enumerater}
\end{proposition}

\begin{proof}
It suffices to prove~\eqref{alcmb}.
The case where either~$a$ or~$b$ is an identity is trivial, so we may assume that neither~$a$ nor~$b$ are identities of~$S$.
A preliminary observation is that $\eps_S(a)\lcm\eps_S(b)$ exists in~$\Um{S}$ if{f}~$\eps_S(a)$ and~$\eps_S(b)$ have a common right multiple in~$\Um{S}$, and, similarly, $a\lcm b$ exists in~$S$ if{f}~$a$ and~$b$ have a common right multiple in~$S$.
This follows from the argument of Dehornoy \cite[Lemma~2.15]{Dis}.

Now suppose that $a\lcm b$ exists in~$S$.
Clearly, $\eps_S(a\lcm b)$ is a common right multiple of~$\eps_S(a)$ and~$\eps_S(b)$.
Let~$\bc$ be a common right multiple of~$\eps_S(a)$ and~$\eps_S(b)$.
It follows from Lemma~\ref{L:AdjepsDelta} that~$a$ and~$b$ are both left divisors of~$\SR{\bc}$; whence $a\lcm b$ exists in~$S$, and is a left divisor of~$\SR{\bc}$.
Again by Lemma~\ref{L:AdjepsDelta}, it follows that $\eps_S(a\lcm b)$ is a left divisor of~$\bc$.
Hence, $\eps_S(a)\lcm\eps_S(b)$ exists in~$\Um{S}$, and is equal to $\eps_S(a\lcm b)$.

Suppose, conversely, that $\eps_S(a)\lcm\eps_S(b)$ exists in~$\Um{S}$; denote this element by~$\bc$.
Since $\eps_S(a)$ and~$\eps_S(b)$ are both left divisors of~$\bc$, it follows from Lemma~\ref{L:AdjepsDelta} that~$a$ and~$b$ are both left divisors of~$\SR{\bc}$; whence~$a\lcm b$ exists in~$S$.
\end{proof}

\begin{remark}\label{Rk:AdjepsDelta}
Lemma~\ref{L:AdjepsDelta} can be interpreted in terms of normal forms and Garside families, for the monoid~$\Um{S}$ (those concepts are defined in Dehornoy~\cite{DDGKM}).
Suppose that~$S$ is a gcd-category, and let $\ba\in\Um{S}$, written as a reduced sequence, say $\ba=\ba_1=(a_1,\dots,a_n)$.
By Lemma~\ref{L:AdjepsDelta}, the element $a_1=\SR{\ba}$ is the largest element of~$S$, with respect to left divisibility, left dividing~$\ba$.
Furthermore, since~$\Um{S}$ is left cancellative (cf. Corollary~\ref{C:PresCancella}), the reduced sequence $\ba_2\eqdef(a_2,\dots,a_n)$ is the only reduced sequence satisfying $\ba=\eps_S(a_1)\cdot\ba_2$, and then the above process can be started again on~$\ba_2$, yielding~$a_2$ as the largest element of~$S$, with respect to left divisibility, left dividing~$\ba_2$; and so on.
At the end of that process, we obtain that the expression $\ba=a_1\cdots a_n$ is the (left) \emph{greedy normal form} of~$\ba$, associated with the Garside family~$\eps_S[S]$ (in one-to-one correspondence with $(S\setminus\Idt S)\cup\set{1}$) of~$\Um{S}$.

An unusual feature of that Garside family is that it is \emph{two-sided}: it is closed under both left and right divisors (any divisor of an element is an element), and also, it is closed under finite lcms (cf. Proposition~\ref{P:Sclslcm}).
The left greedy normal form and the right greedy normal form, with respect to that family, are identical.
\end{remark}

\section{Floating homotopy group of a simplicial complex}
\label{S:FHG}

Although groupoids are instrumental in getting to our results, our main goals are the investigation of certain monoids and groups.
This partly explains why we are less interested in the homotopy groupoid~$\hg{K}$ (where~$K$ is a simplicial complex) than its universal group, which we shall denote by~$\HG{K}$.
In the present section we observe, in particular, that if~$K$ is connected, then~$\HG{K}$ is the free product of the fundamental group~$\pi_1(K,{}_{-})$ by the free group on the edges of a spanning tree of~$K$.
This yields a convenient recipe for quick calculation of floating homotopy groups.
Easy instances of that recipe are given in Examples~\ref{Ex:VG2rGL1} and~\ref{Rk:PosTree}.

\begin{definition}\label{D:FHG}
The \emph{floating homotopy group} of a simplicial complex~$K$, denoted by $\HG{K}$, is the group defined by generators~$[x,y]$ (or $[x,y]_K$ in case~$K$ needs to be specified), where $\set{x,y}\in K^{(1)}$, and relations
 \begin{equation}\label{Eq:rGKrel}
 [x,z]=[x,y]\cdot[y,z]\,,\quad
 \text{for all }x,y,z\in\VR K\text{ such that }
 \set{x,y,z}\in K^{(2)}\,.
 \end{equation}
\end{definition}

An alternative definition of~$\HG{K}$, suggested to the author by a referee, runs as follows.
For a vertex~$o$ not in~$K$, the \emph{cone}~$o*K$ of~$K$ is defined as the simplicial complex consisting of all simplices of~$K$ together with all sets of the form $\set{o}\cup X$ where~$X$ is either empty or a simplex of~$K$.
The edges~$\set{o,x}$, where~$x$ is a vertex of~$K$, define a spanning tree of~$o*K$.
By applying Tietze's Theorem (cf. Rotman \cite[Theorem~11.31]{Rotm1995}), we obtain the relation $\HG{K}\cong\pi_1(o*K,o)$.

%

Heuristically, $\HG{K}$ is the group universally generated by all homotopy classes of paths in~$K$, with endpoints not fixed --- thus our terminology ``floating''.
In fact, it is straightforward to verify the following description of~$\HG{K}$ in terms of the homotopy groupoid Functor~$\hgn$ (cf. Section~\ref{S:Basic}) and the universal group Functor~$\Ugn$ (cf. Section~\ref{S:UnivMon}).

\begin{proposition}\label{P:HG2rG}
$\HG{K}=\Ug{\hg{K}}$, for every simplicial complex~$K$.
\end{proposition}

In particular, for every vertex $p\in\VR{K}$, the fundamental group $\pi_1(K,p)$, consisting of all homotopy classes of all loops at~$p$, is a subgroup (vertex group) of the groupoid $\hg{K}$. 
By Lemma~\ref{L:KerofepsS}, $\pi_1(K,p)$ is thus a subgroup of~$\HG{K}$: the homotopy class~$[\bx]$ of a path $\bx=(x_0,\dots,x_n)$, with $x_0=x_n=p$, is identified with the element $[x_0,x_1]\cdots[x_{n-1},x_n]$ of~$\HG{K}$.

\begin{proposition}\label{P:VG2rGL}
Let~$K$ be a connected simplicial complex, let~$E$ be the set of all edges of a spanning tree of~$K$, and let~$p$ be a vertex of~$K$.
Then $\HG{K}\cong\Fg{E}*\pi_1(K,p)$.
Hence, $\pi_1(K,p)$ is a doubly free factor of~$\HG{K}$.\end{proposition}

\begin{proof}
Denote by~$F$ the free groupoid on the graph~$E$ (cf. Higgins \cite[Chapter~4]{Higg1971}).
By applying Corollary~1, Chapter~12 in Higgins~\cite{Higg1971}, to the fundamental groupoid~$\hg{K}$ of~$K$, we obtain the relation~$\hg{K}\cong F*\pi_1(K,p)$ (where~$*$ denotes the coproduct within the Category of all groupoids).
Since the universal group Functor $\Ugn\colon\GPD\to\GP$ is a left adjoint, it preserves coproducts.
By applying it to the relation above, and since $\Ug{F}\cong\Fg{E}$, we get our result.
\end{proof}

%

\begin{example}\label{Ex:VG2rGL1}
Consider the one-dimensional simplicial complex~$K$ represented in Figure~\ref{Fig:square}.
\begin{figure}[htb]
\includegraphics{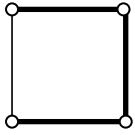}
\caption{A simplicial complex of dimension~$1$}
\label{Fig:square}
\end{figure}
A direct calculation shows that~$\HG{K}$ is a free group on the edges of~$K$.
Hence, $\HG{K}\cong\Fg{4}$.

A spanning tree of~$K$ is represented with thick lines; it has three edges.
By using Tietze's Theorem, we get $\pi_1(K,{}_{-})\cong\ZZ$ (this is trivial in more than one way: for example, the geometric realization of~$K$ is a circle).
We get again $\HG{K}\cong\ZZ*\Fg{3}\cong\Fg{4}$.

See also Example~\ref{Rk:PosTree}.
\end{example}

\section{Interval monoid and interval group of a poset}
\label{S:IntMon}

Every poset~$P$ gives rise, in a standard way, to a category~$\Cat{P}$, formally defined below.
The universal monoid~$\HM{P}$ of~$\Cat{P}$ will be called the \emph{interval monoid} of~$P$, and it is a far more special concept than the universal monoid of an arbitrary category.
By applying the results of Section~\ref{S:divetgcd} to that monoid, we will get a simple criterion for~$\HM{P}$ being a gcd-monoid, always satisfied for the barycentric subdivision of a simplicial complex.
This will enable us, in Theorem~\ref{T:VG2rGL2}, to describe any group in terms of a gcd-monoid of the form~$\HM{P}$, at least up to a doubly free factor.

\begin{definition}\label{D:CatP}
For a poset~$P$, we denote by $\Cat{P}$ the category consisting of all closed intervals $[x,y]$, where $x\leq y$ in~$P$, with (partial) multiplication given by
 \[
 [x,z]=[x,y]\cdot[y,z]\,,\quad\text{whenever }
 x\leq y\leq z\text{ in }P\,.
 \]
\end{definition}

By identifying any $x\in P$ with the singleton interval $[x,x]=\set{x}$, we obtain that the source and target map on~$\Cat{P}$ are given by $\sr[x,y]=x$ and $\tg[x,y]=y$, for all $x,y\in P$ with $x\leq y$.

In the ``object-arrow'' view of~$\Cat{P}$, the objects are the elements of~$P$, with exactly one arrow from~$x$ to~$y$ if $x\leq y$, no arrow if $x\nleq y$.

\begin{definition}\label{D:rGp}
Let~$P$ be a poset.
The \emph{interval monoid} (resp., \emph{interval group}) of~$P$ is defined as $\HM{P}=\Um{\Cat{P}}$ (resp., $\HG{P}=\Ug{\Cat{P}}$).
\end{definition}

It follows that a system of generators and relations for both~$\HM{P}$ (within monoids) and~$\HG{P}$ (within groups) is given as follows:
the generators are (indexed by) the closed intervals $[x,y]$, where $x\leq y$ in~$P$ (write $[x,y]_P$ in case~$P$ needs to be specified), and the defining relations are those of the form
 \begin{align}
 [x,x]&=1\,,
 &&\text{whenever }x\in P\,,\label{Eq:ReflrGP}\\
 [x,z]&=[x,y]\cdot[y,z]\,,
 &&\text{whenever }x\leq y\leq z\text{ within }P\,.
 \label{Eq:TransrGP}
 \end{align}

\begin{remark}\label{Rk:IncAlg}
For a poset~$P$ and a commutative, unital ring~$K$, the semigroup algebra~$K[\HM{P}]$ of the monoid~$\HM{P}$ bears some similarities with the \emph{incidence algebra of~$P$ over~$K$} (the latter concept originates in Rota~\cite{Rota1964a}).
However, the two algebras are not isomorphic as a rule: for example, $y_1\neq x_2$ implies that $[x_1,y_1]\cdot[x_2,y_2]$ vanishes within the incidence algebra of~$P$, but not within~$K[\HM{P}]$.
In particular, if~$K$ is a field and~$P$ is a finite poset with a non-trivial interval $x<y$, the incidence algebra of~$P$ over~$K$ is finite-dimensional, while all powers~$[x,y]^n$, for $n\in\NN$, are distinct in~$\HM{P}$, so~$K[\HM{P}]$ is infinite-dimensional.
\end{remark}

\begin{proposition}\label{P:Meas2Val}
$\HG{P}\cong\HG{\Sim{P}}$, for every poset~$P$.
\end{proposition}

\begin{proof}
The elements $[x,y]_{\Sim{P}}$ of $\HG{\Sim{P}}$, for $x\leq y$ in~$P$, satisfy the defining relations~\eqref{Eq:ReflrGP} and~\eqref{Eq:TransrGP} of~$\HG{P}$,
thus there is a unique group homomorphism $\gf\colon\HG{P}\to\HG{\Sim{P}}$ such that $\gf([x,y]_P)=[x,y]_{\Sim{P}}$ whenever $x\leq y$ in~$P$.
A similar argument, together with an easy case inspection, shows that there is a unique group homomorphism $\psi\colon\HG{\Sim{P}}\to\HG{P}$ such that
 \[
 \psi([x,y]_{\Sim{P}})=\begin{cases}
 [x,y]_P\,,&\text{if }x\leq y\,,\\
 [y,x]_P^{-1}\,,&\text{if }y\leq x\,,
 \end{cases}
 \qquad\text{for 	every chain }\set{x,y}\text{ of }P\,.
 \]
Obviously, $\gf$ and~$\psi$ are mutually inverse.
\end{proof}

By invoking Proposition~\ref{P:VG2rGL}, Proposition~\ref{P:Meas2Val} can be used for quick calculations of~$\HG{P}$ for finite~$P$.
Set $\pi_1(P,p)\eqdef\pi_1(\Sim{P},p)$, whenever $p\in P$.

%

\begin{example}\label{Rk:PosTree}
\begin{figure}[htb]
\includegraphics{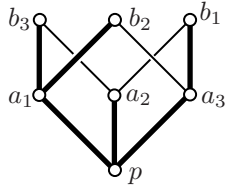}
\caption{A spanning tree in a poset}\label{Fig:P9}
\end{figure}
Denote by~$P$ the poset represented in Figure~\ref{Fig:P9} and denote by~$K$ its chain complex.
Since~$P$ has a lower bound (viz.~$p$), $\pi_1(P,p)$ is trivial.
This can also be verified directly by applying Tietze's Theorem to the spanning tree~$E$ represented in thick lines in Figure~\ref{Fig:P9}.
Since~$E$ has~$6$ elements and by Propositions~\ref{P:VG2rGL} and~\ref{P:Meas2Val}, it follows that $\HG{P}\cong\Fg{6}$.
\end{example}

The following result expresses a quite special property of interval monoids of posets, not shared by all universal monoids of categories as a rule.

\begin{proposition}\label{P:rGPintoGrp}
Let~$P$ be a poset.
Then the interval monoid~$\HM{P}$ embeds into the free group on~$P$.
In particular, the canonical map $\HM{P}\to\HG{P}$ is a monoid embedding.
\end{proposition}

\begin{proof}
Set $\mu(x,y)\eqdef x^{-1}y$, for all $x\leq y$ in~$P$.
Trivially, $\mu(x,x)=1$ and $\mu(x,z)=\mu(x,y)\cdot\mu(y,z)$ whenever $x\leq y\leq z$ in~$P$.
Hence, there is a unique monoid homomorphism $\ol{\mu}\colon\HM{P}\to\Fg{P}$ such that
 \[
 \ol{\mu}([x,y])=x^{-1}y\,,\quad\text{whenever }
 x\leq y\text{ in }P\,.
 \]
It thus suffices to prove that~$\ol{\mu}$ is one-to-one.
Let~$\ba$ be an element of $\HM{P}=\Um{\Cat{P}}$, written in reduced form as
 \begin{equation}\label{Eq:RedFormrGp}
 \ba=[x_1,y_1]\cdot[x_2,y_2]\cdots[x_n,y_n]\,,
 \quad\text{where }x_i<y_i
 \text{ whenever }1\leq i\leq n\,.
 \end{equation}
(cf. Section~\ref{S:UnivMon}).
The fact that the right hand side of~\eqref{Eq:RedFormrGp} is a reduced finite sequence means that $y_i\neq x_{i+1}$ whenever $1\leq i<n$.
Applying the homomorphism~$\ol{\mu}$ to~\eqref{Eq:RedFormrGp}, we obtain
 \begin{equation}\label{Eq:NorminFree}
 \ol{\mu}(\ba)=x_1^{-1}y_1x_2^{-1}y_2\cdots
 x_n^{-1}y_n\,.
 \end{equation}
Since $y_i\neq x_{i+1}$ whenever $1\leq i<n$, the word written on the right hand side of~\eqref{Eq:NorminFree} is the unique reduced word representing the element~$\ol{\mu}(\ba)$ of~$\Fg{P}$.
Hence, $\ol{\mu}(\ba)$ determines the sequence $(x_1,y_1,x_2,y_2,\dots,x_n,y_n)$, and hence it determines~$\ba$, thus completing the proof that~$\ol{\mu}$ is one-to-one.
\end{proof}

Proposition~\ref{P:rGPintoGrp} will get strengthened further in Proposition~\ref{P:rGpLinExt}, where, in the finite case, ``free group'' will be improved to ``free monoid''.

\begin{remark}\label{Rk:rGPintoGrp}
It is known since Bergman \cite[Example~23]{Berg2014} that there are monoids~$M$, embeddable into free groups, of which the universal group is not free.
This remark applies to the monoids~$\HM{P}$: although, by Proposition~\ref{P:rGPintoGrp}, the monoid~$\HM{P}$ always embeds into a free group, it will follow from Theorem~\ref{T:VG2rGL2} that every group is a doubly free factor of some~$\HG{P}$;
and thus, $\HG{P}$ may not be free (e.g., it may have torsion).
\end{remark}

Not every~$\HM{P}$, for~$P$ a poset, is a gcd-monoid.
Nevertheless, the posets~$P$ such that~$\HM{P}$ is a gcd-monoid can be easily characterized.
We set
 \begin{align}
 \INF{P}{a}&=\setm{x\in P}{x\leq a}\,,\label{Eq:INFPa}\\
 \SUP{P}{a}&=\setm{x\in P}{x\geq a}\,,\label{Eq:SUPa}
 \end{align}
for each $a\in P$.

\begin{proposition}\label{P:rGpPgcd}
The following statements hold, for any poset~$P$:
\begin{enumerater}
\item\label{rGpPlgcd}
$\HM{P}$ is a left gcd-monoid if{f} $\SUP{P}{a}$ is a \ms\ for every $a\in P$.

\item\label{rGpPrgcd}
$\HM{P}$ is a right gcd-monoid if{f} $\INF{P}{a}$ is a \js\ for every $a\in P$.

\end{enumerater}
\end{proposition}

\begin{proof}
By symmetry, it suffices to establish~\eqref{rGpPlgcd}.
It follows from Theorem~\ref{T:UmSgcdmon} that the monoid $\HM{P}=\Um{\Cat{P}}$ is a left gcd-monoid if{f} $\Cat{P}$ is a left gcd-category.
By definition, the latter statement means that for all $x,y_1,y_2\in P$, if $x\leq y_1$ and $x\leq y_2$, then the left gcd $[x,y_1]\gcd_{\Cat{P}}[x,y_2]$ is defined.
Now it is straightforward to verify that the left divisibility preordering of~$\Cat{P}$ is given by
 \[
 [u_1,v_1]\dive_{\Cat{P}}[u_2,v_2]\ \Leftrightarrow\ 
 (u_1=u_2\text{ and }v_1\leq v_2)\,,\qquad
 \text{for all }u_1,u_2,v_1,v_2\in P\,.
 \]
In particular, $[x,y_1]\gcd_{\Cat{P}}[x,y_2]$ is defined if{f} $\set{y_1,y_2}$ has a greatest lower bound in $\SUP{P}{x}$.
The desired conclusion follows immediately.
\end{proof}

As an immediate consequence, observe the following.

\begin{corollary}\label{C:rGpPgcd}
Let~$P$ be the barycentric subdivision of a simplicial complex~$K$.
Then $\HM{P}$ is a gcd-monoid.
\end{corollary}

\begin{proof}
The elements of~$P$ are, by definition (cf. Section~\ref{S:Basic}), the simplices of~$K$.
For every simplex~$\ba$ of~$K$, the subset $\INF{P}{\ba}=\setm{\bx\in P}{\bx\subseteq\ba}$ is a \js\ (the join of $\set{\bx_1,\bx_2}$ is $\bx_1\cup\bx_2$; it is nonempty because it contains~$\bx_1$, and it is contained in~$\ba$, thus it is a simplex), and the subset $\SUP{P}{\ba}=\setm{\bx\in P}{\ba\subseteq\bx}$ is a \ms\ (the meet of $\set{\bx_1,\bx_2}$ is $\bx_1\cap\bx_2$; it is nonempty because it contains~$\ba$, and it is contained in~$\bx_1$, thus it is a simplex).
The desired conclusion follows from Proposition~\ref{P:rGpPgcd}.
\end{proof}

The following result shows that loosely speaking, groups and gcd-monoids, and even interval monoids of posets, are objects of similar complexity.

\begin{theorem}\label{T:VG2rGL2}
For every group~$G$, there exists a connected poset~$P$ such that~$G$ is a doubly free factor of~$\HG{P}$ and such that the following conditions hold:
\begin{enumerater}

\item\label{lengthP}
$P$ has height~$2$.

\item\label{rGpPgcd}
$\HM{P}$ is a gcd-monoid.
\end{enumerater}
Furthermore, $P$ can be taken to be finite if{f}~$G$ is finitely presented.
\end{theorem}

\begin{proof}
It is well known (see, for example, Rotman \cite[Theorem~11.64]{Rotm1995}) that~$G$ is isomorphic to the fundamental group of some connected simplicial complex~$K$, which, in addition, is finite in case~$G$ is finitely presented.

Observe that the simplicial complex~$K$ constructed in the proof of \cite[Theorem~11.64]{Rotm1995} has dimension~$2$.
(Alternatively, the fundamental group of~$K$ depends only on the $2$-skeleton~$K^{(2)}$.)

Since~$K$ is a connected simplicial complex, the barycentric subdivision~$P$ of~$K$ (cf. Section~\ref{S:Basic}) is a connected poset.
Since~$K$ has dimension~$2$, $P$ has height~$2$.
Now it is well known that~$K$ and~$\Sim{P}$ have isomorphic fundamental groups:
for example, $K$ and~$\Sim{P}$ have the same geometric realizations (cf. Rotman \cite[Exercise~7.12(i)]{Rotm1988}), and the fundamental group of~$K$ depends only on the geometric realization of~$K$ (cf. Rotman \cite[Theorem~7.36]{Rotm1988};
see also Reynaud~\cite[Section~2]{Reyn2003}).
By Proposition~\ref{P:VG2rGL}, it follows that~$G$ is a doubly free factor of~$\HG{\Sim{P}}$.
By Proposition~\ref{P:Meas2Val}, $\HG{\Sim{P}}=\HG{P}$, thus $G$ is a doubly free factor of~$\HG{P}$.
Moreover, it follows from Corollary~\ref{C:rGpPgcd} that~$\HM{P}$ is a gcd-monoid.

Finally, if~$P$ is finite, then~$\HM{P}$ is finitely presented, thus so is each of its retracts.
Now every free factor is a retract.
\end{proof}

We will see in the next section that if~$P$ is finite (which can be ensured if{f}~$G$ is finitely presented), then the interval monoid~$\HM{P}$ can be embedded into a free monoid.


\section{Embedding gcd-monoids into free monoids}
\label{S:SubMonFM}

In contrast to Proposition~\ref{P:rGPintoGrp}, not every gcd-monoid can be embedded into a free monoid, or even into a free group.
For example, the \emph{braid monoid}~$B_3^+$, defined by the generators~$a$, $b$ and the unique relation $aba=bab$, is a gcd-monoid (cf. Dehornoy \cite[Proposition~IX.1.26]{DDGKM}), but it cannot be embedded into any free group:
for example, any elements~$a$ and~$b$ in a free group such that $aba=bab$ satisfy $(a^2b)^2=(aba)^2$, thus, since free groups are bi-orderable (see, for example, Dehornoy \emph{et al.} \cite[Subsection~9.2.3]{DDRW02}), $a^2b=aba$, and thus $ab=ba$, and hence (since $aba=bab$) $a=b$, which does not hold in~$B_3^+$.

However, for interval monoids the situation is different.

It is obvious that the assignment $P\mapsto\HM{P}$ (resp., $P\mapsto\HG{P}$) defines a Functor, from the Category~$\PO$ of all posets, to the Category~$\MON$ (resp., $\GP$) of all monoids (resp., groups).
More can be said.

\begin{lemma}\label{L:rGp(f)}
Let~$P$ and~$Q$ be posets and let $f\colon P\to Q$ be an isotone map.
Then there is a unique monoid homomorphism $\HM{f}\colon\HM{P}\to\HM{Q}$ such that
 \[
 \HM{f}([x,y]_P)=[f(x),f(y)]_Q\quad\text{whenever }
 x\leq y\text{ in }P\,.
 \]
Moreover, if~$f$ is one-to-one, then so is~$\HM{f}$.
\end{lemma}

\begin{proof}
The existence and uniqueness of~$\HM{f}$ follow trivially from the universal property defining~$\HM{P}$.
Now assume that~$f$ is one-to-one.
Let~$\ba\in\HM{P}$, written in reduced form as
 \[
 \ba=[x_1,y_1]_P\cdot[x_2,y_2]_P\cdots[x_n,y_n]_P\,,
 \]
where each $x_i<y_i$ in~$P$, and $y_i\neq x_{i+1}$ whenever $1\leq i<n$.
Then
 \begin{equation}\label{Eq:redformfa}
 \HM{f}(\ba)=[f(x_1),f(y_1)]_Q\cdot[f(x_2),f(y_2)]_Q
 \cdots[f(x_n),f(y_n)]_Q\,,
 \end{equation}
and since~$f$ is one-to-one, the sequence written on the right hand side of~\eqref{Eq:redformfa} is also reduced.
By the uniqueness of the reduced form of an element of~$\HM{Q}$ (cf. Lemma~\ref{L:x2Ered}), it follows that~$\HM{f}(\ba)$ determines the finite sequence
 \[
 (f(x_1),f(y_1),\dots,f(x_n),f(y_n))\,.
 \]
Since~$f$ is one-to-one, $\HM{f}(\ba)$ also determines the finite sequence
 \[
 (x_1,y_1,\dots,x_n,y_n)\,,
 \]
thus it determines~$\ba$.
\end{proof}

The following easy result, which improves Proposition~\ref{P:rGPintoGrp} in the finite case, is the central part of the present section.

\begin{proposition}\label{P:rGpLinExt}
Let~$P$ be a finite poset.
Then $\HM{P}$ embeds into a \pup{finitely generated} free monoid.
\end{proposition}

\begin{proof}
It is well known that every poset~$P$ has a linear extension~$\widehat{P}$, that is, a linear ordering on the underlying set of~$P$ extending the ordering of~$P$.
By applying Lemma~\ref{L:rGp(f)} to the identity map $f\colon P\to\widehat{P}$, we see that~$\HM{f}$ is a monoid embedding $\HM{P}\hookrightarrow\HM{\widehat{P}}$.
Now denote by~$x_1$, \dots, $x_n$ the elements of~$\widehat{P}$ listed in increasing order.
It is straightforward to verify that~$\HM{\widehat{P}}$ is a free monoid, with free generators $[x_i,x_{i+1}]_{\widehat{P}}$ where $1\leq i<n$.
\end{proof}

\begin{example}\label{Ex:rGpLinExt}
Proposition~\ref{P:rGpLinExt} does not extend to infinite posets.
For example, letting $P=\set{0}\cup\setm{1/n}{n\in\NN\setminus\set{0}}$, the sequence $\vecm{[0,1/n]}{n\in\NN\setminus\set{0}}$ is strictly decreasing in~$\HM{P}$ with respect to the left divisibility ordering.
Hence, $\HM{P}$ cannot be embedded into any free monoid.
\end{example}

As the following example shows, the most natural converse of Proposition~\ref{P:rGpLinExt}, namely whether every finitely presented submonoid of a finitely generated free monoid is an interval monoid, does not hold.

\begin{examplepf}\label{Ex:SubFreenotInt}
A finitely presented submonoid of~$\NN$ which is not the interval monoid of any poset.
\end{examplepf}

\begin{proof}
The submonoid $M=\NN\setminus\set{1}$ of the free monoid~$\NN$ can be defined by the generators~$a$, $b$ and the relations $ab=ba$, $a^3=b^2$
(think of~$a$ and~$b$ as~$2$ and~$3$, respectively, and remember that we are dealing with the \emph{additive} structure of~$\NN$).

Suppose that $M\cong\HM{P}$, for a poset~$P$.

Since the elements~$a$ and~$b$ are atoms of~$M$ (i.e., none of them can be expressed as the sum of two nonzero elements of~$M$), each of them is a standard generator of~$\HM{P}$, that is, $a=[x,y]$ and $b=[x',y']$ for some $x<y$ and $x'<y'$ in~$P$.
Since the finite sequences $([x,y],[x,y],[x,y])$ and $([x',y'],[x',y'])$ are both reduced, and they represent~$a^3$ and~$b^2$, respectively, we get $a^3\neq b^2$, a contradiction.
\end{proof}

\begin{remark}\label{Rk:SubFreenotInt}
The argument of Example~\ref{Ex:SubFreenotInt} can be easily expanded to prove that \emph{A submonoid of~$\NN$ is the interval monoid of a poset if{f} it is cyclic \pup{i.e., of the form $m\NN$ where $m\in\NN$}}.

For those variants of Example~\ref{Ex:SubFreenotInt}, the difference between ``finitely generated'' and ``finitely presented'' is immaterial, because every finitely generated commutative monoid is finitely presented (this is Redei's Theorem, see for example Freyd~\cite{Freyd1968}) and every submonoid of~$\NN$ is finitely generated.
\end{remark}

In the non-commutative case, the situation becomes different.
To begin with, a finitely generated submonoid of a free monoid may not be finitely presented (cf. Spehner \cite[Theorem~3.2]{Speh1977} or \cite[Example~2.9]{Speh1989b}).
Since interval monoids of finite posets are finitely presented, let us focus attention on finitely presented monoids.

\begin{examplepf}\label{Ex:Nongcd}
A finitely presented submonoid of a free monoid, which is also a gcd-monoid, but which is not isomorphic to the interval monoid of any poset.
\end{examplepf}

\begin{proof}
Denote by~$M_6$ the monoid defined by the set of generators $\gS=\set{a,b,c,d,e,f}$ and the relations
 \begin{equation}\label{Eq:PresM6}
 ae=cb\,,\quad da=bf\,.
 \end{equation}
We shall analyze the monoid~$M_6$ by using the tools of Dehornoy \cite[Section~II.4]{DDGKM} (originating in Dehornoy~\cite{Deh2003}).
The presentation~\eqref{Eq:PresM6} is right complemented, which means that for every pair $(s,t)$ of generators, there is at most one relation of the form $sx=ty$ in the presentation.
Moreover, the right complements~$x$ and~$y$ solving this problem are themselves generators.
This makes it very easy to verify the right cube condition (cf. \cite[Subsection~II.4.4]{DDGKM}) on that presentation.
It thus follows from the results of \cite[Section~II.4]{DDGKM} that~$M_6$ is a 
left gcd-monoid.
A symmetric argument yields that~$M$ is a right gcd-monoid.

Further, in order to verify the \emph{right $3$-Ore condition} introduced in Dehornoy~\cite{Dis}, namely that any three elements of~$M_6$, which pairwise admit a right multiple, admit a common right multiple, it suffices (cf. \cite[Section~5]{Dis}) to verify that condition on any subset~$X$ of~$M_6$ containing the atoms and closed under right complementation; in the present case, just take $X=\gS$.
The \emph{left $3$-Ore condition} is verified similarly.
By \cite[Section~4]{Dis}, it follows that~$M_6$ can be embedded into its universal group $G=\Ug{M_6}$.
(The latter fact is also a straightforward application of Adjan's Theorem~\cite{Adj1966}, see also Remmers \cite[Theorem~4.6]{Remm1980}.)

Now denote by~$x$ and~$y$ two new generators and set $\gO=\set{a,b,x,y}$.
Denote by~$M'_6$ the submonoid of~$\Fm{\gO}$ generated by $\gS'=\set{a',b',c',d',e',f'}$ where $a'=a$, $b'=b$, $c'=ax$, $d'=by$, $e'=xb$, $f'=ya$.
Since $a'e'=c'b'$ and $d'a'=b'f'$, there is a unique monoid homomorphism $\eps\colon M_6\to M'_6$ sending each generator of~$\Sigma$ to its primed version (i.e., $\eps(a)=a'$, and so on).

The following claim shows that~$M_6$ embeds into the free monoid~$\Fm{\gO}$.

\begin{sclaim}
The map~$\eps$ is an isomorphism from~$M_6$ onto~$M'_6$. 
\end{sclaim}

\begin{scproof}
Within the group $G=\Ug{M_6}$, we set $\ol{x}=a^{-1}c=eb^{-1}$ and $\ol{y}=b^{-1}d=fa^{-1}$.
We denote by $\gf\colon\Fm{\gO}\to G$ the unique monoid homomorphism such that $\gf(a)=a$, $\gf(b)=b$, $\gf(x)=\ol{x}$, and $\gf(y)=\ol{y}$, and we denote by~$\eta$ the restriction of~$\gf$ to~$M'_6$.
Then $\eta(a')=a$, $\eta(b')=b$, $\eta(c')=a\ol{x}=c$, $\eta(d')=b\ol{y}=d$, $\eta(e')=\ol{x}b=e$, and $\eta(f')=\ol{y}a=f$, whence~$\eta\circ\eps$ is the inclusion map $M_6\hookrightarrow G$.
In particular, $\eps$ is one-to-one.
But~$\eps$ is, by construction, surjective.
\end{scproof}

Therefore, in order to conclude the proof of Example~\ref{Ex:Nongcd}, it suffices to prove that~$M_6$ is not isomorphic to the interval monoid of any poset.
Let $M_6\cong\HM{P}$, for a poset~$P$.
Since the elements of~$\gS$ are all atoms of~$M_6$, they are all standard generators of~$\HM{P}$.
Since $ae=cb$ and $a\neq c$, none of the finite sequences $(a,e)$ and $(c,b)$ is reduced, thus there are elements $o,p,q,i\in P$ such that $o<p<i$ and $o<q<i$ within~$P$, and $a=[o,p]$, $e=[p,i]$, $c=[o,q]$, $b=[q,i]$.
In particular, $\min a=o<q=\min b$.
A similar argument, applied to the relation $bf=da$, yields the relation $\min b<\min a$; a contradiction.
\end{proof}

\section{Gcd-monoids arising from extreme spindles}
\label{S:Spindle}

In this section we introduce a class of categories, associated to certain intervals in posets, closely related to the categories~$\Cat{P}$.
Those categories are a source of counterexamples in the  paper Dehornoy and Wehrung~\cite{DW1}; in particular, as we will see in that paper, the universal monoid of such a category, although cancellative, may not embed into any group.

Those categories (cf. Definition~\ref{D:SpdCat}) are built from the following concept.

\begin{definition}\label{D:Spindle}
Let~$P$ be a poset and let $u<v$ in~$P$,
with $[u,v]$ of height $\geq2$ (i.e., $u<z<v$ for some~$z$).
We say that the closed interval~$[u,v]$ is a \emph{spindle} of~$P$ if the comparability relation on the open interval $\oo{u,v}\eqdef\setm{x\in P}{u<x<v}$ is an equivalence relation.
If, in addition, $u$ is a minimal element of~$P$ and $v$ is a maximal element of~$P$,
we say that $[u,v]$ is an \emph{extreme spindle} of~$P$.
\end{definition}

In what follows, we shall denote by~$\cC_{u,v}$ the set of all maximal chains of~$[u,v]$.

\begin{proposition}\label{P:SpindleChar}
The following are equivalent, for any closed interval $[u,v]$, of height~$\geq2$, in a poset~$P$:
\begin{enumeratei}
\item\label{OrSpdDef}
$[u,v]$ is a spindle of~$P$.

\item\label{ChSpdDef}
Any two distinct maximal chains of $[u,v]$ meet in $\set{u,v}$.
\end{enumeratei}
\end{proposition}

\begin{proof}
Denote by~$\sim$ the comparability relation on $\oo{u,v}$.

\eqref{OrSpdDef}$\Rightarrow$\eqref{ChSpdDef}.
Assume~\eqref{OrSpdDef} and set $\Sigma=\oo{u,v}/{\sim}$.
Every element of~$\Sigma$ is a chain and $\oo{u,v}$ is the disjoint union of~$\Sigma$.
It follows that $[u,v]=\bigcup\vecm{X\cup\set{u,v}}{X\in\Sigma}$.
The sets $X\cup\set{u,v}$, where $X\in\Sigma$, are exactly the maximal chains of $[u,v]$, and they meet pairwise at $\set{u,v}$.

\eqref{ChSpdDef}$\Rightarrow$\eqref{OrSpdDef}.
By assumption, the elements of~$\cC_{u,v}$ are pairwise meeting at $\set{u,v}$, thus every $x\in\oo{u,v}$ belongs to a unique $C(x)\in\cC_{u,v}$.
It follows that $x\sim y$ if{f} $C(x)=C(y)$, for all $x,y\in\oo{u,v}$.
The desired conclusion follows easily.
\end{proof}

\begin{definition}\label{D:SpdCat}
For any extreme spindle $[u,v]$ in a poset~$P$, we endow the set
 \[
 \Cat{P,u,v}\eqdef\pI{\Cat{P}\setminus\set{[u,v]}}\cup\cC_{u,v}\,,
 \]
with the (partial) multiplication given by
 \begin{align}
 [u,u]\cdot X=X\cdot[v,v]&=X\,,&&
 \text{whenever }X\in\cC_{u,v}\,,\label{Eq:XneutralMC}\\
 [x,y]\cdot[y,z]&=[x,z]\,,&&
 \text{whenever }x\leq y\leq z\text{ and }
 (x,z)\neq(u,v)\,,\label{Eq:xyyzxzC}\\
 [u,z]\cdot[z,v]&=Z\,,&&
 \text{whenever }u<z<v\,,\ Z\in\cC_{u,v}\,,
 \text{ and }z\in Z\,.\label{Eq:uxxvCx}
 \end{align}
\end{definition}

The verification of the following technical lemma is tedious, but straightforward, and we omit its proof.

\begin{lemma}\label{L:SpindleCat}
Let $[u,v]$ be an extreme spindle in a poset~$P$.
Then $\Cat{P,u,v}$, endowed with the multiplication given in Definition~\textup{\ref{D:SpdCat}}, is a category.
Furthermore, the divisibility orderings~$\dive$ and~$\divet$, on that category, are given by
 \begin{align*}
 [x,y]\dive[x',y']&\quad\text{if{f}}\quad
 x=x'\text{ and }y\leq y'\,,
 &&\text{whenever }[x,y],[x',y']\neq[u,v]\,,\\
 Z\ndive[x,y]&\,,
 &&\text{whenever }[x,y]\neq[u,v]
 \text{ and }Z\in\cC_{u,v}\,,\\
 [x,y]\dive Z&\text{ if{f} }x=u\text{ and }y\in Z\,,
 &&\text{whenever }[x,y]\neq[u,v]
 \text{ and }Z\in\cC_{u,v}\,,\\
 X\dive Y&\text{ if{f} }X=Y\,,
 &&\text{whenever }X,Y\in\cC_{u,v}\,,\\
 \intertext{and}
 [x,y]\divet[x',y']&\quad\text{if{f}}\quad
 x'\leq x\text{ and }y=y'\,,
 &&\text{whenever }[x,y],[x',y']\neq[u,v]\,,\\
 Z\ndivet[x,y]&\,,
 &&\text{whenever }[x,y]\neq[u,v]
 \text{ and }Z\in\cC_{u,v}\,,\\
 [x,y]\divet Z&\text{ if{f} }y=v\text{ and }x\in Z\,,
 &&\text{whenever }[x,y]\neq[u,v]
 \text{ and }Z\in\cC_{u,v}\,,\\
 X\divet Y&\text{ if{f} }X=Y\,,
 &&\text{whenever }X,Y\in\cC_{u,v}\,.\\
  \end{align*}
\end{lemma}

By identifying any $x\in P$ with the singleton interval $[x,x]=\set{x}$, we obtain that the source and target map, on~$\Cat{P,u,v}$, are given by
 \begin{align*}
 \sr[x,y]=x&
 \text{ and }\tg[x,y]=y\,,&&\text{whenever }[x,y]\neq[u,v]\,,\\
 \sr Z=u&
 \text{ and }\tg Z=v\,,&&\text{whenever }Z\in\cC_{u,v}\,.
 \end{align*}

\begin{proposition}\label{P:SpindleCat}
Let~$P$ be a poset such that $\HM{P}$ is a gcd-monoid, and let $[u,v]$ be an extreme spindle of~$P$.
Then $\Cat{P,u,v}$ is a gcd-category, and its universal monoid $\HM{P,u,v}\eqdef\Um{\Cat{P,u,v}}$ is a gcd-monoid.
\end{proposition}

\begin{proof}
It follows from Proposition~\ref{P:rGpPgcd} that $\INF{P}{a}$ is a \js\ and $\SUP{P}{a}$ is a \ms, for every $a\in P$.
Then a direct application of Lemma~\ref{L:SpindleCat} yields that
$\Cat{P,u,v}$ is a gcd-category, where the left gcd (resp., right gcd), of a pair of elements with the same source (resp., target), are respectively given by
 \begin{align*}
 [x,y]\gcd[x,z]&=[x,y\wedge z]\,,
 &&\text{for all }[x,y],[x,z]\neq[u,v]\,,\\
 [u,x]\gcd Z&=[u,\max([u,x]\cap Z)]\,,
 &&\text{whenever }u\leq x\,,\ x\neq v\,,\text{ and }
 Z\in\cC_{u,v}\,,\\
 [x,z]\gcdt[y,z]&=[x\vee y,z]\,,
 &&\text{for all }[x,z],[y,z]\neq[u,v]\,,\\
 [x,v]\gcdt Z&=[\min([x,v]\cap Z),v]\,,
 &&\text{whenever }x\leq v\,,\ x\neq u\,,
 \text{ and }Z\in\cC_{u,v}\,,\\
 X\gcd Y&=X\gcdt Y=\set{u}\,,
 &&\text{for all }X\neq Y\text{ in }\cC_{u,v}\,.
 \end{align*}
By Theorem~\ref{T:UmSgcdmon}, it follows that $\HM{P,u,v}$ is a gcd-monoid.
\end{proof}

It is easy to construct examples showing that the assumption of Proposition~\ref{P:SpindleCat}, which is equivalent to saying that each $\INF{P}{a}$ is a \js\ and each $\SUP{P}{a}$ is a \ms, can be relaxed.
For example, defining~$P$ by its covering relations $0<p_i$ and $p_i<q_j$, whenever $i,j\in\set{0,1}$, with the extreme spindle $[0,q_0]$, it is easy to verify that $\Cat{P,0,q_0}$ is a gcd-category.
Nevertheless, $\SUP{P}{0}=P$ is not a \ms, because $\set{q_0,q_1}$ has no greatest lower bound.

\begin{proposition}\label{P:SpdMonPres}
Let~$P$ be a poset and let $[u,v]$ be an extreme spindle of~$P$.
Then the monoid $\HM{P,u,v}$ can be defined by the generators $[x,y]$, where $x<y$ in~$P$ with $(x,y)\neq(u,v)$, and the relations
 \begin{equation}\label{Eq:SpdEq}
 [x,z]=[x,y]\cdot[y,z]\,,\quad\text{whenever }x<y<z\text{ in }P
 \text{ with }(x,z)\neq(u,v)\,.
 \end{equation}
\end{proposition}

\begin{proof}
We verify that $M_0\eqdef\HM{P,u,v}$ satisfies the universal property defining the monoid presented by the relations~\eqref{Eq:SpdEq}.
Thus let~$M$ be a monoid, with elements~$a_{x,y}\in M$, for $(x,y)\neq[u,v]$, such that $a_{x,z}=a_{x,y}a_{y,z}$ whenever $x<y<z$ and $(x,z)\neq(u,v)$.
We need to prove that there is a unique monoid homomorphism $\gf\colon M_0\to M$ such that $\gf([x,y])=a_{x,y}$ whenever $x<y$ and $(x,y)\neq(u,v)$.

We first extend the function $(x,y)\mapsto a_{x,y}$ by setting $a_{x,x}=1$ whenever $x\in P$.

For the existence part, it suffices to prove that~$M_0$ is generated by the subset $\Cat{P}\setminus\set{[u,v]}$.
Let~$Z$ be a maximal chain of $[u,v]$.
Since $[u,v]\neq\set{u,v}$, there is $z\in Z$ such that $u<z<v$.
Hence, within $\Cat{P,u,v}$ (thus within~$M_0$), $Z=[u,z]\cdot[z,v]$, as required.
This completes the proof of the uniqueness part.

Let us deal with existence now.
We claim that for any maximal chain~$Z$ of~$[u,v]$, the element $a_{u,z}a_{z,v}$, where $z\in Z\setminus\set{u,v}$, is independent of~$z$.
Indeed, any $x,y\in Z\setminus\set{u,v}$ are comparable, say $x\leq y$, and then
$a_{u,x}a_{x,v}=a_{u,x}a_{x,y}a_{y,v}=a_{u,y}a_{y,v}$,
thus proving our claim.
Denote by~$b_Z$ the common value of all $a_{u,z}a_{z,v}$, where $z\in Z\setminus\set{u,v}$.
We need to prove that the~$a_{x,y}$, where $[x,y]\neq[u,v]$, and the~$b_Z$, where $Z\in\cC_{u,v}$, satisfy the relations defining the monoid~$M_0$, which are also the relations \eqref{Eq:XneutralMC}--\eqref{Eq:uxxvCx} defining the category $\Cat{P,u,v}$.
The only non-trivial instances to be verified are $a_{x,z}=a_{x,y}a_{y,z}$ whenever $x\leq y\leq z$ and $(x,z)\neq(u,v)$, and $b_Z=a_{u,z}a_{z,v}$ whenever $Z\in\cC_{u,v}$ and $z\in Z\setminus\set{u,v}$, which all hold by construction.
\end{proof}

\section{A criterion of group-embeddability for universal monoids}
\label{S:Pos2UMG}

For any category~$S$, the universal property defining~$\Um{S}$ implies immediately that there is a unique monoid homomorphism $\gf_S\colon\Um{S}\to\Ug{S}$ such that $\gf_S\circ\eps_S=\eta_S$.
In fact, $\gf_S=\eta_{\Um{S}}$.
Since there are monoids (even cancellative ones) that cannot be embedded into groups (see Meakin~\cite{Meak2007} for a survey), $\gf_S$ may not be an embedding.
However, the result below states a convenient
criterion for this to occur.
It states that for any category~$S$, the embeddability of~$\Um{S}$ into a group can be verified ``locally'', that is, on the hom-sets.
The main trick used in the proof of Theorem~\ref{T:EmbUmS2Grp} will be called the \emph{highlighting expansion} of a morphism.

\begin{theorem}\label{T:EmbUmS2Grp}
The following are equivalent, for any category~$S$:
\begin{enumeratei}
\item\label{UmSintogrp1}
The canonical map~$\gf_S\colon\Um{S}\to\Ug{S}$ is one-to-one.

\item\label{UmSintogrp2}
$\Um{S}$ embeds into a group.

\item\label{IdSepEmb}
There are a group~$G$ and a functor $\psi\colon S\to G$ such that the restriction of~$\psi$ to every hom-set of~$S$ is one-to-one.
\end{enumeratei}
\end{theorem}

\begin{proof}
\eqref{UmSintogrp1}$\Rightarrow$\eqref{UmSintogrp2} is trivial, while \eqref{UmSintogrp2}$\Rightarrow$\eqref{IdSepEmb} follows immediately from Lemma~\ref{L:KerofepsS}.

Finally let us prove \eqref{IdSepEmb}$\Rightarrow$\eqref{UmSintogrp1}.
Let~$G$ and~$\psi$ be as prescribed in~\eqref{IdSepEmb}.
We need to prove that~$\gf_S$ is one-to-one.
Our proof will be an amplification of the one of Proposition~\ref{P:rGPintoGrp}.

The map $ \psi'\colon S\to\Fg{\Idt S}*G$,
$x\mapsto(\sr{x})^{-1}\psi(x)\tg{x}$ is a functor.
It works by adding to~$\psi(x)$ the endpoints information on~$x$, thus we will call it the \emph{highlighting expansion} of~$\psi$.

\begin{sclaim}
The restriction of~$\psi'$ to every hom-set of~$S$ is one-to-one.
Furthermore,
$\psi'(x)=1$ implies that $x\in\Idt S$, for any $x\in S$.
\end{sclaim}

\begin{scproof}
Let $a,b\in\Idt S$ and let $x,y\in S(a,b)$ such that $\psi'(x)=\psi'(y)$.
This means that $a^{-1}\psi(x)b=a^{-1}\psi(y)b$ within $\Fg{\Idt S}*G$, that is, $\psi(x)=\psi(y)$, so, by assumption, $x=y$.

Now let $x\in S$ such that $\psi'(x)=1$.
Setting $a=\sr{x}$ and $b=\tg{x}$, this means that
$a^{-1}\psi(x)b=1$ within $\Fg{\Idt S}*G$, so, by the uniqueness of the normal form for elements of that group (see, for example, Rotman \cite[Theorem~11.52]{Rotm1995}), $a=b$ and $\psi(x)=1$, and so $\psi(x)=\psi(a)$, where~$x$ and~$a$ both belong to $S(a,a)$.
By our assumption, it follows that $x=a$.
\end{scproof}

By the Claim above, we may replace~$(G,\psi)$ by its highlighting expansion\newline $\pI{\Fg{\Idt S}*G,\psi'}$ and thus assume from the start that
 \begin{equation}\label{Eq:psisep1}
 \psi(x)=1\text{ implies that }x\in\Idt S\,,\quad
 \text{for each }x\in S\,.
 \end{equation}
Forming again the highlighting expansion~$\psi'$ of that new map~$\psi$, it follows from the universal property defining~$\Ug{S}$ that there is a unique group homomorphism $\gs\colon\Ug{S}\to\Fg{\Idt S}*G$ such that $\psi'=\gs\circ\eta_S$.

Now let $\bx,\by\in\Um{S}$ such that $\gf_S(\bx)=\gf_S(\by)$.
Writing $\red{\bx}=(x_1,\dots,x_m)$ and $\red{\by}=(y_1,\dots,y_n)$, this can be written
 \[
 \eta_S(x_1)\cdots\eta_S(x_m)=
 \eta_S(y_1)\cdots\eta_S(y_n)
 \quad\text{within }G\,.
 \]
By applying the homomorphism~$\gs$, we get
 \[
 \psi'(x_1)\cdots\psi'(x_m)=
 \psi'(y_1)\cdots\psi'(y_n)\quad
 \text{within }\Fg{\Idt S}*G\,,
 \]
which, writing $(a_i,b_i)=(\sr{x_i},\tg{x_i})$ and $(c_j,d_j)=(\sr{y_j},\tg{y_j})$, means that
 \begin{multline*}
 a_1^{-1}\psi(x_1)b_1a_2^{-1}\psi(x_2)b_2
 \cdots a_m^{-1}\psi(x_m)b_m=
 c_1^{-1}\psi(y_1)d_1c_2^{-1}\psi(y_2)d_2
 \cdots c_n^{-1}\psi(y_n)d_n\\
 \text{within }\Fg{\Idt S}*G\,.
 \end{multline*}
Since each $b_i\neq a_{i+1}$ (because $x_ix_{i+1}\uparrow$) and $c_j\neq d_{j+1}$ (because $y_jy_{j+1}\uparrow$), and since, by~\eqref{Eq:psisep1}, each~$\psi(x_i)$ and each~$\psi(y_j)$ belongs to~$G\setminus\set{1}$, it follows from the uniqueness of the normal form for elements of~$\Fg{\Idt S}*G$ that $m=n$, each $(a_i,b_i)=(c_i,d_i)$, and each $\psi(x_i)=\psi(y_i)$.
Since~$x_i$ and~$y_i$ both belong to~$S(a_i,b_i)$, it follows from our assumption that $x_i=y_i$.
Therefore, $\bx=\by$.
\end{proof}

By defining~$\psi$ as the constant map with value~$1$, we get immediately part of the result of  Proposition~\ref{P:rGPintoGrp} --- namely, that the interval monoid of a poset always embeds into a group.
As the following example shows, the range of application of Theorem~\ref{T:EmbUmS2Grp} goes beyond universal monoids of ordered sets.

\begin{example}\label{Ex:EmbGrp1}
Denote by~$C_6$ the monoid defined by the generators~$a$, $b$, $c$, $a'$, $b'$, $c'$ and the relations
 \begin{equation}\label{Eq:PresC6}
 ab'=ba'\,,\quad bc'=cb'\,,\quad ac'=ca'\,.
 \end{equation}
We prove, with the help of Theorem~\ref{T:EmbUmS2Grp}, that~$C_6$ can be embedded into a group.

We consider distinct symbols~$0$, $1$, $2$, $a$, $b$, $c$, $a'$, $b'$, $c'$, $\ol{a}$, $\ol{b}$, $\ol{c}$ and we define the category~$S$, with objects~$0$, $1$, $2$ and nonempty hom-sets defined by $S(i,i)=\set{i}$ whenever $i\in\set{0,1,2}$, $S(0,1)=\set{a,b,c}$,
$S(1,2)=\set{a',b',c'}$, $S(0,2)=\set{aa',bb',cc',\ol{a},\ol{b},\ol{c}}$, and composition defined by
 \[
 ab'=ba'=\ol{c}\,,\quad bc'=cb'=\ol{a}\,,\quad
 ac'=ca'=\ol{b}\,.
 \]
By definition, $C_6=\Um{S}$.
It is straightforward to verify that~$S$ is a gcd-category.
Hence, by Theorem~\ref{T:UmSgcdmon}, $C_6$ is a gcd-monoid.

Denote by $\psi\colon S\to\ZZ^3$ the unique functor such that $\psi(a)=\psi(a')=(1,0,0)$, $\psi(b)=\psi(b')=(0,1,0)$, $\psi(c)=\psi(c')=(0,0,1)$.
It is straightforward to verify that~$\psi$ is one-to-one on every hom-set of~$S$.
For example, the elements~$\psi(aa')$, $\psi(bb')$, $\psi(cc')$, $\psi(\ol{a})$, $\psi(\ol{b})$, $\psi(\ol{c})$ are all distinct.
Since $S(i,i)$ is a singleton for each $i\in\set{0,1,2}$, it follows from Theorem~\ref{T:EmbUmS2Grp} that~$C_6$ embeds into a group.

It is interesting to analyze~$C_6$ \emph{via} the methods of Dehornoy \cite[Section~II.4]{DDGKM}.
The presentation~\eqref{Eq:PresC6} is both left and right complemented, and it satisfies the left and right cube conditions, which yields another proof that~$C_6$ is a gcd-monoid.
On the other hand, the right $3$-Ore condition fails in~$C_6$, for~$a$, $b$, $c$ pairwise admit common 
right-multiples, but they admit no global common right-multiple.
Hence the methods of \cite[Section~II.4]{DDGKM} are \emph{a priori} not sufficient to infer the result, established above, that~$C_6$ embeds into its universal group.

Then comes one more surprise.
Eliminating, in that order, the variables~$c'$ and~$b'$ from~\eqref{Eq:PresC6} yields $c'=b^{-1}cb'$, then $b'=c^{-1}ba^{-1}ca'$, and then $ac^{-1}b=bc^{-1}a$.
Therefore, the universal group~$G_6$ of~$C_6$ can be defined by generators~$a$, $b$, $c$, $a'$ and the unique relation $ac^{-1}b=bc^{-1}a$.
Changing~$c$ to~$c^{-1}$, we obtain the alternate presentation of~$G_6$, with generators~$a$, $b$, $c$, $a'$, and relation $acb=bca$.
Hence, $G_6$ is also the universal group of the monoid~$D_4$ defined by generators~$a$, $b$, $c$, $a'$, and relation $acb=bca$.
By applying to that presentation the methods of Dehornoy \cite[Section~II.4]{DDGKM}, it can be verified that~$D_4$ is a n\oe therian gcd-monoid \emph{satisfying both left and right $3$-Ore conditions}.
By applying Dehornoy \cite[Section~4]{Dis}, this gives another proof that~$D_4$ embeds into its universal group.
However, the embeddability of~$D_4$ into its universal group would not be sufficient, \emph{a priori}, to infer the above result that~$C_6$ embeds into its universal group, because the change of presentation described above involves changing~$c$ to~$c^{-1}$, which requires an ambient group.
\end{example}


\begin{thebibliography}{10}

\bibitem{Adj1966}
Serge{\u{\i}}~I. Adjan, \emph{Defining relations and algorithmic problems for
  groups and semigroups}, Trudy Mat. Inst. Steklov. \textbf{85} (1966), 123.
  \MR{0204501}

\bibitem{Berg2014}
George~M. Bergman, \emph{On monoids, {$2$}-firs, and semifirs}, Semigroup Forum
  \textbf{89} (2014), no.~2, 293--335. \MR{3258484}

\bibitem{BriSai1972}
Egbert Brieskorn and Kyoji Saito, \emph{Artin-{G}ruppen und
  {C}oxeter-{G}ruppen}, Invent. Math. \textbf{17} (1972), 245--271.
  \MR{0323910}

\bibitem{Deh2003}
Patrick Dehornoy, \emph{Complete positive group presentations}, J. Algebra
  \textbf{268} (2003), no.~1, 156--197. \MR{2004483}

\bibitem{DDGKM}
\bysame, \emph{Foundations of {G}arside {T}heory}, EMS Tracts in Mathematics,
  vol.~22, European Mathematical Society (EMS), Z\"urich, 2015, with
  Fran{\c{c}}ois Digne, Eddy Godelle, Daan Krammer and Jean Michel.
  \MR{3362691}

\bibitem{Dis}
\bysame, \emph{Multifraction reduction {I}: {T}he 3-{O}re case and
  {A}rtin-{T}its groups of type {FC}}, J. Comb. Algebra \textbf{1} (2017),
  no.~2, 185--228. \MR{3634782}

\bibitem{DDRW02}
Patrick Dehornoy, Ivan Dynnikov, Dale Rolfsen, and Bert Wiest, \emph{Why are
  {B}raids {O}rderable?}, Panoramas et Synth\`eses [Panoramas and Syntheses],
  vol.~14, Soci\'et\'e Math\'ematique de France, Paris, 2002. \MR{1988550}

\bibitem{DW1}
Patrick Dehornoy and Friedrich Wehrung, \emph{Multifraction reduction {III}:
  {T}he case of interval monoids}, J. Comb. Algebra \textbf{1} (2017), no.~4,
  341--370. \MR{3713055}

\bibitem{Freyd1968}
Peter Freyd, \emph{Redei's finiteness theorem for commutative semigroups},
  Proc. Amer. Math. Soc. \textbf{19} (1968), 1003. \MR{0227290}

\bibitem{Gars1969}
Francis~A. Garside, \emph{The braid group and other groups}, Quart. J. Math.
  Oxford Ser. (2) \textbf{20} (1969), 235--254. \MR{0248801}

\bibitem{Higg1971}
Philip~J. Higgins, \emph{Notes on {C}ategories and {G}roupoids}, Van Nostrand
  Reinhold Co., London-New York-Melbourne, 1971, Van Nostrand Reinhold
  Mathematical Studies, No. 32. \MR{0327946}

\bibitem{Meak2007}
John Meakin, \emph{Groups and semigroups: connections and contrasts}, Groups
  {S}t. {A}ndrews 2005. {V}ol. 2, London Math. Soc. Lecture Note Ser., vol.
  340, Cambridge Univ. Press, Cambridge, 2007, pp.~357--400. \MR{2331597}

\bibitem{Remm1980}
John~H. Remmers, \emph{On the geometry of semigroup presentations}, Adv. in
  Math. \textbf{36} (1980), no.~3, 283--296. \MR{577306}

\bibitem{Reyn2003}
Eric Reynaud, \emph{Algebraic fundamental group and simplicial complexes}, J.
  Pure Appl. Algebra \textbf{177} (2003), no.~2, 203--214. \MR{1954333
  (2003i:55022)}

\bibitem{Rota1964a}
Gian-Carlo Rota, \emph{On the foundations of combinatorial theory. {I}.
  {T}heory of {M}\"obius functions}, Z. Wahrscheinlichkeitstheorie und Verw.
  Gebiete \textbf{2} (1964), 340--368 (1964). \MR{0174487}

\bibitem{Rotm1988}
Joseph~J. Rotman, \emph{An {I}ntroduction to {A}lgebraic {T}opology}, Graduate
  Texts in Mathematics, vol. 119, Springer-Verlag, New York, 1988. \MR{957919}

\bibitem{Rotm1995}
\bysame, \emph{An {I}ntroduction to the {T}heory of {G}roups}, fourth ed.,
  Graduate Texts in Mathematics, vol. 148, Springer-Verlag, New York, 1995.
  \MR{1307623}

\bibitem{Speh1977}
Jean-Claude Spehner, \emph{Pr\'esentations et pr\'esentations simplifiables
  d'un mono\"\i de simplifiable}, Semigroup Forum \textbf{14} (1977), no.~4,
  295--329. \MR{0466363}

\bibitem{Speh1989b}
\bysame, \emph{Every finitely generated submonoid of a free monoid has a finite
  {M}al\cprime cev's presentation}, J. Pure Appl. Algebra \textbf{58} (1989),
  no.~3, 279--287. \MR{1004608}

\end{thebibliography}

\providecommand{\noopsort}[1]{}\def\cprime{$'$}
  \def\polhk#1{\setbox0=\hbox{#1}{\ooalign{\hidewidth
  \lower1.5ex\hbox{`}\hidewidth\crcr\unhbox0}}}
\providecommand{\bysame}{\leavevmode\hbox to3em{\hrulefill}\thinspace}
\providecommand{\MR}{\relax\ifhmode\unskip\space\fi MR }
\providecommand{\MRhref}[2]{%
  \href{http://www.ams.org/mathscinet-getitem?mr=#1}{#2}
}
\providecommand{\href}[2]{#2}

\end{document}